\documentclass[10pt,11pt]{amsart}%
\usepackage[dvips]{graphicx}
\usepackage{amsmath}
\usepackage{color}
\usepackage{amsfonts}
\usepackage{amssymb}
\usepackage[margin=.6in]{geometry}
\usepackage{tikz}
\usepackage{young}%
\providecommand{\U}[1]{\protect\rule{.1in}{.1in}}
\usetikzlibrary{matrix,arrows,decorations.pathmorphing,automata, matrix, positioning, calc, shapes.multipart}
\providecommand{\U}[1]{\protect\rule{.1in}{.1in}}
\providecommand{\U}[1]{\protect\rule{.1in}{.1in}}
\providecommand{\U}[1]{\protect\rule{.1in}{.1in}}
\providecommand{\U}[1]{\protect\rule{.1in}{.1in}}
\providecommand{\U}[1]{\protect\rule{.1in}{.1in}}
\providecommand{\U}[1]{\protect\rule{.1in}{.1in}}

\newtheorem{theorem}{Theorem}[section]

\newtheorem{corollary}[theorem]{Corollary}

\newtheorem{definition}[theorem]{Definition}
\newtheorem{example}[theorem]{Example}

\newtheorem{lemma}[theorem]{Lemma}

\newtheorem{proposition}[theorem]{Proposition}
\newtheorem{remark}[theorem]{Remark}

\begin{document}
\title[Keys and Demazure crystals]{Keys and Demazure crystals for Kac-Moody algebras\\ }
\date{October, 2019}
\author{Nicolas Jacon and C\'{e}dric Lecouvey}
\address{N.J.: Universit\'{e} de Reims Champagne-Ardennes, UFR Sciences exactes et
naturelles Laboratoire de Math\'{e}matiques FRE 2011 Moulin de la Housse BP
1039 51100 REIMS.\\
C. L.: Institut Denis Poisson UMR CNRS 7013 Facult\'{e} des Sciences et
Techniques, Universit\'{e} Fran\c{c}ois Rabelais Parc de Grandmont}

\begin{abstract}
The Key map is an important tool in the determination of the Demazure crystals
associated to Kac-Moody algebras. In finite type A, it can be computed in the
tableau realization of crystals by a simple combinatorial procedure due to
Lascoux and Sch\"{u}tzenberger.\ We show that this procedure is a part of a
more general construction holding in the Kac-Moody case that we illustrate in
finite types and affine type A.  In affine type A, we introduce higher level
generalizations of core partitions which are expected to play an important role 
 in the representation theory of Ariki-Koike algebras.

\end{abstract}
\maketitle

\section{Introduction}

Kac-Moody algebras are infinite-dimensional analogues of semisimple Lie
algebras. Their classification is based on the notion of Cartan datum, a
generalization of the finite root systems. In particular, a Kac-Moody algebra
$\mathfrak{g}$ admits an enveloping algebra $U(\mathfrak{g})$, a Weyl group
$W$, a weight lattice $P$ and a cone $P_{+}$ of dominant weights. To each
dominant weight $\lambda$ is associated a highest weight $U(\mathfrak{g}%
)$-module $V(\lambda)$. The works of Kashiwara, Lusztig and Littelmann during
the 90's have shown the existence of a fundamental object associated to
$V(\lambda)$: the crystal $B(\lambda)$. It is an oriented graph whose
combinatorics encodes many informations on $V(\lambda)$. In particular, it is
endowed with a weight function with values in $P$ whose generating series over
$B(\lambda)$ coincides with the character of $V(\lambda)$ (see \cite{kash} and
the references therein). The graph $B(\lambda)$ admits a unique source vertex
$b_{\lambda}$ (its highest weight vertex) and there is a simple action of the
Weyl group $W$ on $B(\lambda)$. Also for $\lambda,\mu$ in $P_{+}$, the crystal
$B(\lambda)\otimes B(\mu)=\{b\otimes b^{\prime}\mid b\in B(\lambda),b^{\prime
}\in B(\mu)\}$ of the tensor product $V(\lambda)\otimes V(\mu)$ can be easily
computed from $B(\lambda)$ and $B(\mu)$.\ In particular, $b_{\lambda}\otimes
b_{\mu}$ is of highest weight $\lambda+\mu$ in $B(\lambda)\otimes B(\mu
)$.\ The crystals with highest weight vertices $b_{\lambda}\otimes b_{\mu}$
and $b_{\mu}\otimes b_{\lambda}$ in $B(\lambda)\otimes B(\mu)$ and
$B(\mu)\otimes B(\lambda)$ are then isomorphic. The corresponding isomorphism
can be regarded as the restriction of more general isomorphisms between
$B(\lambda)\otimes B(\mu)$ and $B(\mu)\otimes B(\lambda)$ called combinatorial
$R$-matrices.

The Demazure modules $V(\lambda)_{w}$ are $U^{+}(\mathfrak{g})$-submodules of
$V(\lambda)$ defined for any $w\in W$. Quite remarkably, each such Demazure
module $V(\lambda)$ also admits a crystal $B_{w}(\lambda)$ which is a subgraph
of $B(\lambda)$. It has been proved by Littelmann that the generating series
of the weight function over $B_{w}(\lambda)$ gives the Demazure character of
$B_{w}(\lambda)$. Given the crystal $B(\lambda)$, it is a natural question to
ask whether a vertex $b$ in $B(\lambda)$ belongs to a Demazure crystal
$B_{w}(\lambda)$. This problem may be solved by using a combinatorial
procedure which involves the computation of a certain map called the right Key
map. This map associates to each vertex $b$ of $B(\lambda)$ an element
$K^{R}(b)$ in the orbit $O(\lambda)$ of $b_{\lambda}$ under the action of $W$.
The Key map can be computed in any realization of the abstract crystal
$B(\lambda)$ but has a great combinatorial complexity.\ Observe also that the
algebra $U^{+}(\mathfrak{g})$ admits a crystal $B(\infty)$ with Demazure
crystals $B_{w}(\infty),w\in W$ and associated Key maps.

In finite type $A$ (i.e. for the Lie algebras $\mathfrak{sl}_{n}$), the
dominant weights $\lambda$ can be regarded as partitions and each crystal
$B(\lambda)$ has a simple realization in terms of semistandard tableaux of
shape $\lambda$.\ In \cite{Las1}, Lascoux and Sch\"{u}tzenberger defined a
simple procedure associating to such a tableau a \textquotedblleft
Key\textquotedblright\ tableau defined as a semistandard tableau such that
each column of height $h$ is included in any column of height $h^{\prime}\geq
h$. They then showed that these Key tableaux permit to compute the Demazure
characters. By using the Littelmann path model and the dilatation of crystals
introduced by Kashiwara, one can then prove that the Key map defined in
\cite{Las1} can be recover from the previous general definition when the
crystals $B(\lambda)$ are realized in terms of semistandard tableaux.

\bigskip

The main goal of this paper is to give a general reduction procedure to
compute the Key map for any Kac-Moody algebra. Our strategy is to show that
the approach of Lascoux and Sch\"{u}tzenberger can be generalized to any
crystal $B(\lambda)$ associated to any Kac-Moody algebra. More precisely, we
explain how the Key map $K^{R}$ can be computed for any weight $\lambda$,
recursively on $\lambda$ essentially by reduction to the case of the
fundamental weights. In this perspective, the Demazure crystals can be
characterized by the Key map for the fundamental weights, the previous
restrictions of combinatorial $R$-matrices and the description of the strong
Bruhat order on cosets of $W$. In particular, in finite type $A$, the Key map
for a fundamental weight is the identity, the combinatorial $R$-matrices can
be computed on tableaux by the Jeu de Taquin procedure and the strong Bruhat
order is easy to describe. Thus one recovers the results of \cite{Las1}. For
the classical types and for type $G_{2}$, there are analogue simple tableaux
models and we then illustrate our general procedure by giving natural
extension of Lascoux-Sch\"{u}tzenberger's construction. They might also be
adapted to the remaining exceptional cases based on the \textquotedblleft
tableaux\textquotedblright\ existing model for crystals (see \cite{BS}).\ This
suggests that recent results by Brubaker and al. \cite{B3G}, Masson \cite{Mas}
and Proctor \cite{Proct} for type $A$ might have generalizations in finite
types. Note that we were informed during the redaction of this paper that
Santos \cite{San} also simultaneously got the description of the Key in type
$C$. His approach, based on the symplectic plactic monoid, is nevertheless
distinct from ours. It is also worth mentioning that the Key map can be
computed as the last direction of paths in the alcove path model \cite{Len1}
and there exist crystal isomorphisms \cite{LenLub} between this model and the
tableaux model of Kashiwara and Nakashima.\ We next focus on the affine type
$A_{e-1}^{(1)}$ for which there also exists an interesting crystal model using
multipartitions and related to the modular representation theory of
Ariki-Koike algebras (some generalizations of the Hecke algebras). When
$\lambda=\omega_{i}$ is a fundamental weight, $O(\omega_{i})$ is parametrized
by particular partitions called $e$-cores and the Key map can be computed
thanks to a combinatorial procedure introduced in \cite{A}. Also the strong
Bruhat order on $O(\omega_{i})$ corresponds to the inclusion of the Young
diagrams of the $e$-cores. Thus, we can apply the previous reduction. Along
the way, we introduce higher level generalizations of the core partitions
which give interesting analogues of the Young lattice and which 
 have a nice interpretation in terms of the block theory of  Ariki-Koike algebras
\cite{JLw}. Let us conclude by mentioning there are also quite simple
combinatorial models for the highest weight crystals in any affine type (see
for example \cite{HK}).\ It would be interesting to have a combinatorial
description of the key maps and the $R$-matrices for the fundamental weights
in this setting.

\bigskip

The paper is organized as follows. Section 2 is a recollection of basics facts
on crystals and Demazure characters.\ In Section 3, we present the previous
recursive procedure to compute the Key map.\ We explain how it can be used for
Demazure crystals associated to finite types in Section 4. The affine type $A$
case is studied in Section 5 where we introduce the notion of
$(e,\boldsymbol{s})$-core as a natural labelling of the orbit of the empty
multipartition in Uglov's and Kleshchev realizations of crystals.\ We also
describe the Key map on Kleshchev multipartitions.\ Finally, we explain how
our results on the the Demazure subcrystals in $B(\lambda)$ can be used to
characterize the Demazure subcrystals in $B(\infty)$.

\section{Background on Keys and Demazure crystals}

\subsection{Crystals for integrable modules over Kac-Moody algebras}

\subsubsection{Background on root systems and Kac-Moody algebras}

\label{subsec_KMA}Let $I$ be a finite set and $A=(a_{i,j})_{(i,j)\in I^{2}}$
be a generalized Cartan matrix of rank $r$.\ This means that the entries of
the matrix satisfy the following conditions

\begin{enumerate}
\item $a_{i,j}\in\mathbb{Z}$ for $i,j\in I^{2},$

\item $a_{i,i}=2$ for $i\in I^{2},$

\item $a_{i,j}=0$ if and only if $a_{j,i}=0$ for $i,j\in I^{2}$.
\end{enumerate}

We will also assume that $A$ is indecomposable: given subsets $I$ and $J$ of
$\{1,\ldots,n\}$, there exists $(i,j)\in I^{2}$ such that $a_{i,j}\not =%
0$.\ We refer to \cite{KacB} for the classification of indecomposable
generalized Cartan matrices.\ Recall there exist only three kinds of such
matrices: when all the principal minors of $A$ are positive, $A$ is of finite
type and corresponds to the Cartan matrix of a simple Lie algebra over
$\mathbb{C}$; when all the proper principal minors of $A$ are positive and
$\det(A)=0$ the matrix $A$ is said of\emph{ }affine type; otherwise $A$ is of
indefinite type. For technical reasons, from now on, we will restrict
ourselves to symmetrizable generalized Cartan matrices i.e. we will assume
there exists a diagonal matrix $D$ with entries in $\mathbb{Z}_{>0}$ such that
$DA$ is symmetric.

The root and weight lattices associated to a generalized symmetrizable Cartan
matrix are defined by mimicking the construction for the Cartan matrices of
finite type.\ Let $P^{\vee}$ be a free abelian group of rank $2\left\vert
I\right\vert -r$ with $\mathbb{Z}$-basis $\{h_{i}\mid i\in I\}\cup
\{d_{1},\ldots,d_{\left\vert I\right\vert -r}\}$.\ Set $\mathfrak{h}:=P^{\vee
}\otimes_{\mathbb{Z}}\mathbb{C}$ and $\mathfrak{h}_{\mathbb{R}}:=P^{\vee
}\otimes_{\mathbb{Z}}\mathbb{R}$. The weight lattice $P$ is then defined by
\[
P:=\{\gamma\in\mathfrak{h}^{\ast}\mid\gamma(P^{\vee})\subset\mathbb{Z\}}%
\text{.}%
\]
Set $\Pi^{\vee}:=\{h_{i}\mid i\in I\}$.$\;$One can then choose a set
$\Pi:=\{\alpha_{i}\mid i\in I\}$ of linearly independent vectors in
$P\subset\mathfrak{h}^{\ast}$ such that $\alpha_{i}(h_{j})=a_{i,j}$ for
$i,j\in I^{2}$ and $\alpha_{i}(d_{j})\in\{0,1\}$ for $i\in\{1,\ldots
,\left\vert I\right\vert -r\}$.\ The elements of $\Pi$ are the simple
roots.\ The free abelian group $Q:=\bigoplus_{i=1}^{\left\vert I\right\vert
}\mathbb{Z}\alpha_{i}$ is the root lattice.\ The quintuple $(A,\Pi,\Pi^{\vee
},P,P^{\vee})$ is called a generalized\emph{ }Cartan datum associated to the
matrix $A$.\ Let $P_{+}=\{\lambda\in P\mid\lambda(h_{i})\geq0$ for any $i\in
I\}$ be the set of dominant weights. For any $i\in I$, the\emph{ }fundamental
weight $\omega_{i}\in P$ is such that $\omega_{i}(h_{j})=\delta_{i,j}$ for
$j\in I$ and $\omega_{i}(d_{j})=0$ for $j\in\{1,\ldots,\left\vert I\right\vert
-r\}$.\ 

For any $i\in I$, we define the simple reflection $s_{i}$ on $\mathfrak{h}%
^{\ast}$ by
\begin{equation}
s_{i}(\gamma)=\gamma-h_{i}(\gamma)\alpha_{i}\ \text{for any }\gamma\in
P\text{.} \label{defSi}%
\end{equation}
The Weyl group $W$ is the subgroup of $GL(\mathfrak{h}^{\ast})$ generated by
the reflections $s_{i}$.\ This is a Coxeter group acting on the weight lattice
$P$ and we refer the reader to \cite{BB} for a complete exposition. In
particular, all the reduced decompositions of a fixed $w\in W$ have the same
length $\ell(w)$. In the sequel we shall need the following characterizations
of the strong Bruhat order $\trianglelefteq$ and the weak Bruhat order $\leq$
on $W$. Given $u$ and $v$ in $W$, we have

\begin{itemize}
\item $u\trianglelefteq v$ if and only if every reduced decomposition of $v$
admits a subword that is a reduced decomposition of $u$.

\item $u\leq v$ if and only if there are reduced decompositions of $u$ and $v$
such that $u$ is a suffix of $v$.
\end{itemize}

Of course, if $u\leq v$, we have $u\trianglelefteq v$ but the converse is not
true in general. For any dominant weight $\lambda$, write $W_{\lambda}$ for
the stabilizer of $\lambda$ under the action of $W$. Every $w\in W$ then
admits a unique decomposition on the form $w=p_{\lambda}(w)v$ with $v\in
W_{\lambda}$ and $p_{\lambda}(w)\in W$ of minimal length.\ Let us denote by
$W^{\lambda}$ the image of $W$ by the projection map $p_{\lambda}$. By setting
$J_{\lambda}=\{i\in I\mid s_{i}(\lambda)=\lambda\}$, we get that $u$ belongs
to $W^{\lambda}$ if and only if none of its reduced decompositions ends with a
generator $s_{i}$ such that $i\in J_{\lambda}$ (alternatively all its reduced
expressions ends with a generator $s_{i},i\notin J_{\lambda}$). Finally recall
that for any $w$ and $w^{\prime}$ in $W$, we have%
\begin{equation}
w\trianglelefteq w^{\prime}\Longrightarrow p_{\lambda}(w)\trianglelefteq
p_{\lambda}(w^{\prime}). \label{proj}%
\end{equation}
We have in fact the more precise lemma (which follows from Theorem $2.6.1$ in
\cite{BB})

\begin{lemma}
\label{Lem_SBOCosets}Assume $\lambda,\mu$ are dominant weights and
$(w,w^{\prime})\in W^{\lambda+\mu}\times W^{\lambda+\mu}$.\ Then
\[
w\trianglelefteq w^{\prime}\Longleftrightarrow\left\{
\begin{array}
[c]{c}%
p_{\lambda}(w)\trianglelefteq p_{\lambda}(w^{\prime}),\\
p_{\mu}(w)\trianglelefteq p_{\mu}(w^{\prime}).
\end{array}
\right.
\]

\end{lemma}

\bigskip

Let $\mathfrak{g}$ be the symmetrizable Kac-Moody algebra associated to the
generalized Cartan matrix $A$.\ We yet refer to \cite{KacB} for a detailed
definition of $\mathfrak{g}$ and write as usual $R$ its root system and $P$
its weight lattice. The algebra $\mathfrak{g}$ admits a presentation by
relations on its Chevalley type generators $e_{i},f_{i},$ $i\in I$ and $h\in
P^{\ast}$.\ There exists a relevant semisimple category $\mathcal{O}%
_{\mathrm{int}}$ of integrable $\mathfrak{g}$-modules whose simple are
parametrized by the dominant weights in $P_{+}$.\ To each $\lambda\in P_{+}$
corresponds a unique (up to isomorphism) irreducible highest weight integrable
$\mathfrak{g}$-module $V(\lambda)$ of highest weight $\lambda$.\ The
irreducible module $V(\lambda)$ decomposes into weight spaces $V(\lambda
)=\bigoplus_{\gamma\in P}V(\lambda)_{\gamma};$ and each weight space
$V(\lambda)_{\gamma}$ is finite-dimensional.\ Consider the ring algebra
$\mathbb{Z}[P]$ with basis the formal exponentials $e^{\beta},\beta\in P$.\ We
have an action of $W$ on $\mathbb{Z}[P]$ defined by $w\cdot e^{\beta
}=e^{w(\beta)}$.\ Set $\mathbb{Z}^{W}[P]=\{X\in\mathbb{Z}[P]\mid
w(X)=X\}$.\ The character $s_{\lambda}$ of $V(\lambda)$ is the element of
$\mathbb{Z}[P]$ defined by $s_{\lambda}:=\sum_{\gamma\in P}K_{\lambda,\gamma
}e^{\gamma}$ where $K_{\lambda,\gamma}:=\operatorname{dim}(V(\lambda)_{\gamma
})$. It belongs in fact to $\mathbb{Z}^{W}[P]$ because $K_{\lambda,\gamma
}=K_{\lambda,w(\gamma)}$ for any $w\in W$. We have the Weyl-Kac character
formula: for any $\lambda\in P_{+}$,%
\begin{equation}
s_{\lambda}=\frac{\sum_{w\in{\mathsf{W}}}\varepsilon(w)e^{w(\lambda+\rho
)-\rho}}{\prod_{\alpha\in R_{+}}(1-e^{-\alpha})^{m_{\alpha}}} \label{WCF}%
\end{equation}
where $m_{\alpha}$ is the multiplicity of the roots $\alpha$ (equal to $1$ in
the finite case).

The quantum group $U_{q}(\mathfrak{g})$ is also defined from the same
generalized Cartan matrix $A$. It also admits a presentation by generators and
relations which can be regarded as $q$-deformation of that of $\mathfrak{g}$
(see \cite{kash}). Roughly speaking, one obtains the enveloping algebra
$U(\mathfrak{g)}$ of $\mathfrak{g}$ as the limit of $U_{q}(\mathfrak{g)}$ when
$q$ tends to $1$. This implies that the representation theory of
$U_{q}(\mathfrak{g)}$ is essentially similar to that of $U(\mathfrak{g)}$ and
thus also to that of $\mathfrak{g}$.\ Therefore, for simplicity and since we
do not need to distinguish the different module structures in the sequel, we
will use the same notation for the category of integrable modules of
$\mathfrak{g,}U(\mathfrak{g})$ and $U_{q}(\mathfrak{g)}$. In particular, for
each dominant weight $\lambda$, there exists a unique $U_{q}(\mathfrak{g}%
)$-module in $\mathcal{O}_{\mathrm{int}}$ also denoted by $V(\lambda)$.

\subsubsection{Crystals of integrable modules}

To each dominant weight $\lambda$ corresponds a crystal graph $B(\lambda)$
which can be regarded as the combinatorial skeleton of the simple module
$V(\lambda)$. Its structure can be defined from the notion of canonical bases
as introduced  by Lusztig \cite{Luca} and subsequently studied by Kashiwara
under the name of global bases (see \cite{kash} and \cite{kash2}). It also has
a purely combinatorial definition in terms of Littelmann's path model (see
\cite{Lit1}).\ The crystal $B(\lambda)$ is a graph whose set of vertices is
endowed with a weight function $\mathrm{wt}:B(\lambda)\rightarrow P$ and with
the structure of a colored and oriented graph given by the action of the
crystal operators $\tilde{f}_{i}$ and $\tilde{e}_{i}$ with $i\in I$. More
precisely, we have an oriented arrow $b\overset{i}{\rightarrow}b^{\prime}$
between two vertices $b$ and $b^{\prime}$ in $B(\lambda)$ if and only if
$b^{\prime}=\tilde{f}_{i}(b)$ or equivalently $b=\tilde{e}_{i}(b^{\prime})$.
We have $\tilde{f}_{i}(b)=0$ (resp. $\tilde{e}_{i}(b)=0$) when no arrow $i$
starts from $b$ (resp. ends at $b$). There is a unique vertex $b_{\lambda}$ in
$B(\lambda)$ such that $\tilde{e}_{i}(b_{\lambda})=0$ for any $i\in I$ called
the highest weight vertex of $B(\lambda)$ and we have $\mathrm{wt}(b_{\lambda
})=\lambda$.\ Thus, for any $b\in B(\lambda)$, there is a path $b=\tilde
{f}_{i_{1}}\cdots\tilde{f}_{i_{r}}(b_{\lambda})$ from $b_{\lambda}$ to
$b$.\ The weight function $\mathrm{wt}$ is such that
\[
\mathrm{wt}(b)=\lambda-\sum_{k=1}^{r}\alpha_{i_{k}}.
\]
For any $i\in I$, the crystal $B(\lambda)$ decomposes into $i$-chains.\ For
any vertex $b\in B(\lambda)$, set $\varphi_{i}(b)=\max\{k\mid\tilde{f}_{i}%
^{k}(b)\neq0\}$ and $\varepsilon_{i}(b)=\max\{k\mid\tilde{e}_{i}^{k}%
(b)\neq0\}$.\ We have
\[
\mathrm{wt}(\tilde{f}_{i}(b))=\mathrm{wt}(b)-\alpha_{i}\text{ and }s_{\lambda
}=\sum_{b\in B(\lambda)}e^{\mathrm{wt}(b)}.
\]
The Weyl group $W$ acts on the vertices of $B(\lambda)$: the action of the
simple reflection $s_{i}$ on $B(\lambda)$ sends each vertex $b$ on the unique
vertex $b^{\prime}$ in the $i$-chain of $b$ such that $\varphi_{i}(b^{\prime
})=\varepsilon_{i}(b)$ and $\varepsilon_{i}(b^{\prime})=\varphi_{i}(b)$ for
any $i\in I$. This simply means that $b$ and $b^{\prime}$ correspond by the
reflection with respect to the center of the $i$-chain containing $b$. We
shall write
\[
O(\lambda)=\{w\cdot b_{\lambda}=b_{w\lambda}\mid w\in W\}
\]
for the orbit of the highest weight vertex of $B(\lambda)$.\ Observe
$b_{w\lambda}$ is then the unique vertex in $B(\lambda)$ of weight $w\lambda$.

More generally, the crystal $B_{M}$ of any module $M$ in $\mathcal{O}%
_{\mathrm{int}}$ is the disjoint union of the crystals associated to the
irreducible modules appearing in its decomposition.\ In particular, the
multiplicity of the irreducible module $V(\lambda)$ in $M$ corresponds to the
number of copies of the crystal $B(\lambda)$ in $B_{M}$. Consider $M$ and $N$
two modules in $\mathcal{O}_{\mathrm{int}}$ with crystals $B_{M}$ and $B_{N}$,
respectively. The crystal associated to $M\otimes N$ is the crystal
$B_{M}\otimes B_{N}$ whose set of vertices is the direct product of the sets
of vertices of $B_{M}$ and $B_{N}$ and whose crystal structure is given by the
following rules\footnote{Observe our convention here is not the same as in
\cite{kash} and \cite{kash2}.}
\begin{equation}
\tilde{e}_{i}(u\otimes v)=\left\{
\begin{array}
[c]{l}%
u\otimes\tilde{e}_{i}(v)\text{ if }\varepsilon_{i}(u)\leq\varphi_{i}(v)\\
\tilde{e}_{i}(u)\otimes v\text{ if }\varepsilon_{i}(u)>\varphi_{i}(v)
\end{array}
\right.  \text{ and }\tilde{f}_{i}(u\otimes v)=\left\{
\begin{array}
[c]{l}%
\tilde{f}_{i}(u)\otimes v\text{ if }\varphi_{i}(v)\leq\varepsilon_{i}(u)\\
u\otimes\tilde{f}_{i}(v)\text{ if }\varphi_{i}(v)>\varepsilon_{i}(u)
\end{array}
\right.  . \label{tens_crys}%
\end{equation}

A crystal $B(\infty)$ for the positive part $U_{q}^{+}(\mathfrak{g)}$ of the
quantum group $U_{q}(\mathfrak{g})$ is also available by the results of
Lusztig \cite{Luca} and Kashiwara \cite{kash}. This crystal $B(\infty)$ admits
a unique source vertex $b_{\emptyset}$.\ Moreover, for any $\lambda\in P_{+}$,
there exists a unique embedding of crystals $\pi_{\lambda}:B(\lambda
)\hookrightarrow B(\infty)$ so that for any path $b=\tilde{f}_{i_{1}}%
\cdots\tilde{f}_{i_{r}}(b_{\lambda})$ in $B(\lambda)$, we have $\pi_{\lambda
}(b)=\tilde{f}_{i_{1}}\cdots\tilde{f}_{i_{r}}(b_{\emptyset})$ in $B(\infty)$.

\subsubsection{Cosets of the Weyl group and crystals}

\label{subsub_reduc_strongBO}There is a one-to-one correspondence between
$W^{\lambda}$ and $O(\lambda)$ which associates to each $w\in W^{\lambda}$ the
vertex $b_{w\lambda}$.\ Also in $O(\lambda)$, the vertices $b_{w\lambda}$ and
$b_{s_{i}w\lambda}$ are such that%
\begin{equation}
\left\{
\begin{array}
[c]{c}%
b_{s_{i}w\lambda}=\tilde{f}_{i}^{\varphi_{i}(b_{w\lambda})}b_{w\lambda}\text{
with }\varepsilon_{i}(b_{w\lambda})=0\text{ if }\ell(s_{i}w)=\ell(w)+1,\\
b_{s_{i}w\lambda}=\tilde{e}_{i}^{\varepsilon_{i}(b_{w\lambda})}b_{w\lambda
}\text{ with }\varphi_{i}(b_{w\lambda})=0\text{ if }\ell(s_{i}w)=\ell(w)-1.
\end{array}
\right.  \label{i-chain_orbit}%
\end{equation}
In particular, for any element $b_{w\lambda}\in O(\lambda)$ and any $i\in I$,
we have either $\varepsilon_{i}(b_{w\lambda})=0$, or $\varphi_{i}(b_{w\lambda
})=0$\footnote{Nevertheless, the condition $\varepsilon_{i}(b_{w\lambda})=0$
or $\varphi_{i}(b_{w\lambda})=0$ does not characterize the elements of
$\mathcal{O}(\lambda)$.}. Now if $w$ belongs to $W^{\lambda}$ with reduced
decomposition $w=s_{i_{1}}\cdots s_{i_{\ell}}$, we will have
\[
b_{w\lambda}=\tilde{f}_{i_{\ell}}^{a_{\ell}}\cdots\tilde{f}_{i_{1}}^{a_{1}%
}(b_{\lambda})=s_{i_{\ell}}\cdots s_{i_{1}}\cdot b_{\lambda}%
\]
with $a_{1}=\varphi_{i_{1}}(b_{\lambda})$ and $a_{k}=\varphi_{i_{k}}(\tilde
{f}_{i_{k-1}}^{a_{k-1}}\cdots\tilde{f}_{i_{1}}^{a_{1}}b_{\lambda})$ for
$k=2,\ldots,\ell$. The converse is true which permits to identify the reduced
expressions of $w=s_{i_{\ell}}\cdots s_{i_{1}}\in W^{\lambda}$ with the
directed paths in $O(\lambda)$ from $b_{\lambda}$ to $b_{w\lambda}$. In some
sense, the crystal $B(\lambda)$ can be regarded as an automaton which
associates to any $w\in W$, its projection $p_{\lambda}(w)$ on $W^{\lambda}$.
On can also observe that $O(\lambda)$ has the structure of the Hasse diagram
on $W^{\lambda}$ by putting arrows $b_{w\lambda}\dashrightarrow b_{s_{i}%
w\lambda}$ when $\ell(s_{i}w)=\ell(w)+1$.

\subsubsection{Dilatation of crystals}

Consider a positive integer $m$ and $\lambda$ a dominant weight. There exists
a unique embedding of crystals $\psi_{m}:B(\lambda)\hookrightarrow
B(m\lambda)$ such that for any vertex $b\in B(\lambda)$ and any path
$b=\tilde{f}_{i_{1}}\cdots\tilde{f}_{i_{k}}(b_{\lambda})$ in $B(\lambda),$ we
have
\[
\psi_{m}(b)=\tilde{f}_{i_{1}}^{m}\cdots\tilde{f}_{i_{k}}^{m}(b_{m\lambda}).
\]
Since the vertex $b_{\lambda}^{\otimes m}$ is of highest weight $m\lambda$ in
$B(\lambda)^{\otimes m}$, one gets a particular realization $B(b_{\lambda
}^{\otimes m})$ of $B(m\lambda)$ in $B(\lambda)^{\otimes m}$ with highest
weight vertex $b_{\lambda}^{\otimes m}$. This thus gives a canonical
embedding
\begin{equation}
K_{m}:\left\{
\begin{array}
[c]{c}%
B(b_{\lambda})\hookrightarrow B(b_{\lambda}^{\otimes m})\subset B(b_{\lambda
})^{\otimes m}\\
b\longmapsto b_{1}\otimes\cdots\otimes b_{m}%
\end{array}
\right.  \label{embdedd}%
\end{equation}
Consider $\lambda$ and $\mu$ two dominant weights. Write $b_{\lambda}$ and
$b_{\mu}$ for the highest weight vertices of $B(\lambda)$ and $B(\mu)$. Then
$B(b_{\lambda}\otimes b_{\mu})$ is a realization of the abstract crystal
$B(\lambda+\mu)$ and we can define the $m$-dilatation $K_{m}:B(b_{\lambda
}\otimes b_{\mu})\hookrightarrow B(b_{\lambda+\mu}^{\otimes m})$ as in
(\ref{embdedd}). The following lemma shows there is another natural
$m$-dilatation of $B(b_{\lambda}\otimes b_{\mu})$ (see for example Corollary
2.1.3 in \cite{Lec2} for a proof).

\begin{lemma}
\label{Lem_Kmtild}The map%
\[
K_{m}^{\prime}:\left\{
\begin{array}
[c]{c}%
B(b_{\lambda}\otimes b_{\mu})\hookrightarrow B(b_{\lambda}^{\otimes m}\otimes
b_{\mu}^{\otimes m})\\
b_{1}\otimes b_{2}\longmapsto K_{m}(b_{1})\otimes K_{m}(b_{2})
\end{array}
\right.
\]
is a $m$-dilatation of $B(b_{\lambda}\otimes b_{\mu})$, that is for any $i\in
I$ we have
\[
K_{m}^{\prime}(\tilde{f}_{i}(b_{1}\otimes b_{2}))=\tilde{f}_{i}^{m}%
K_{m}^{\prime}(b_{1}\otimes b_{2})\text{ and }K_{m}^{\prime}(\tilde{e}%
_{i}(b_{1}\otimes b_{2}))=\tilde{e}_{i}^{m}K_{m}^{\prime}(b_{1}\otimes
b_{2}).
\]

\end{lemma}

\begin{theorem}
\label{Th_Dila}(see \cite{kash2})

\begin{enumerate}
\item For any $w\in W,$ we have $K_{m}(b_{w\lambda})=b_{w\lambda}^{\otimes m}$.

\item Consider $b\in B(\lambda)$. When $m$ has sufficiently many factors,
there exist elements $w_{1},\ldots,w_{m}$ in $W$ such that $K_{m}%
(b)=b_{w_{1}\lambda}\otimes\cdots\otimes b_{w_{m}\lambda}$. Moreover, in this case

\begin{enumerate}
\item up to repetition, the elements $b_{w_{1}\lambda}$ and $b_{w_{m}\lambda}$
in $K_{m}(b)$ do not then depend on $m$,

\item the sequence $(w_{1}\lambda,\ldots,w_{m}\lambda)$ in $K_{m}(b)$ does not
depend on the realization of the crystal $B(\lambda)$ and we have
$w_{1}\trianglelefteq\cdots\trianglelefteq w_{m}$.
\end{enumerate}
\end{enumerate}
\end{theorem}

From Assertion 2 of the previous theorem, we can define the left and right
Keys of an element in $B(\lambda)$.

\begin{definition}
Let $b\in B(\lambda)$, then the left Key $K^{L}(b)$ of $b$ and the right Key
$K^{R}(b)$ of $b$ are defined as follows:
\[
K^{L}(b)=b_{w_{1}\lambda}\text{ and }K^{R}(b)=b_{w_{m}\lambda}.
\]

\end{definition}

\begin{remark}
\ \label{Rq_impo}

\begin{enumerate}
\item By Assertion 4 of the theorem, the sequence $(w_{1}\lambda,\ldots
,w_{m}\lambda)$ does not depend on the realization of the crystal $B(\lambda
)$.\ Nevertheless, the components $b_{w_{k}\lambda}$ do and thus also the left
and right Keys.

\item Assume $B_{1}(\lambda)$ and $B_{2}(\lambda)$ are two realizations of the
crystal $B(\lambda)$ and $\phi:B_{1}(\lambda)\rightarrow B_{2}(\lambda)$ the
associated crystal isomorphism.\ Let $K_{m}^{(1)}$ and $K_{m}^{(2)}$ be the
crystal embedding defined from $B_{1}(\lambda)$ and $B(\lambda)$ as in
(\ref{embdedd}). Since $\phi$ is a crystal isomorphism and $K_{m}^{(1)}%
,K_{m}^{(2)}$ are both crystal embeddings we have $K_{m}^{(2)}\circ\phi
=\phi^{\otimes m}\circ K_{m}^{(1)}$ where $\phi^{\otimes m}$ is defined on
$B_{1}(\lambda)^{\otimes m}$ by applying $\phi$ to each factors. In
particular, for any $b\in B_{1}(\lambda)$ we have
\begin{equation}
K^{L}\cdot\phi(b)=\phi\cdot K^{L}(b)\text{ and }K^{R}\cdot\phi(b)=\phi\cdot
K^{R}(b).\label{K_comm_fi}%
\end{equation}

\end{enumerate}
\end{remark}

Following Kashiwara, let us now define for any $\mu\in W\cdot\lambda$, the
set
\[
\overline{B}_{\mu}(\lambda)=\{b\in B(\lambda)\mid K^{R}(b)=b_{\mu}\}
\]
We then have $B(\lambda)=%
{\textstyle\bigsqcup\limits_{\mu\in W\cdot\lambda}}
\overline{B}_{\mu}(\lambda)$.

\subsection{Crystals of Demazure modules}

Let $\lambda$ be a dominant weight and consider $w\in W$.\ Then, there exists
(up to a constant) a unique highest weight vector $v_{w\lambda}$ in
$V(\lambda)$. The Demazure module associated to $v_{w\lambda}$ is the
$U_{q}^{+}(\mathfrak{g})$-module defined by
\[
D_{w}(\lambda):=U_{q}^{+}(\mathfrak{g})\cdot v_{w\lambda}.
\]
Demazure \cite{Dem} introduced the character $s_{\lambda}^{w}$ of
$D_{w}(\lambda)$ and shows that it can be computed by applying to $e^{\lambda
}$ a sequence of divided difference operators given by any decomposition of
$w$. More precisely, define for any $i\in I$ the operator $D_{i}$ on
$\mathbb{Z}[P]$ by%
\[
D_{i}(X)=\frac{X-e^{-\alpha_{i}}(s_{i}\cdot X)}{1-e^{-\alpha_{i}}}.
\]
Consider a reduced decomposition $w=s_{i_{1}}\cdots s_{i_{\ell}}$ of $w$.
Then, Demazure proved that $D_{w}=D_{i_{1}}\cdots D_{i_{\ell}}$ depends only
on $w$ and not on the reduced decomposition considered.\ Then $s_{\lambda}%
^{w}=D_{w}(e^{\lambda})\in\mathbb{Z}[P]$ is the Demazure character. Later
Kashiwara \cite{Kash0} and Littelmann \cite{Lit1} defined a relevant notion of
crystals for the Demazure modules. To do this, consider for any $w\in W$, the
set
\[
\overline{B}_{w\lambda}(\lambda)=\{b\in B(\lambda)\mid K^{R}(b)=b_{w\lambda
}\}.
\]
By definition we have $\overline{B}_{w\lambda}(\lambda)=\overline
{B}_{w^{\prime}\lambda}(\lambda)$ when $w$ and $w^{\prime}$ belong to the same
left coset of $W/W_{\lambda}$. We also get $B(\lambda)=%
{\textstyle\bigsqcup\limits_{w\lambda\in W\lambda}}
\overline{B}_{w\lambda}(\lambda)$.

\begin{definition}
The Demazure crystal $B_{w}(\lambda)$ is defined by
\begin{equation}
B_{w}(\lambda)=%
{\textstyle\bigsqcup\limits_{w^{\prime}\trianglelefteq w}}
\overline{B}_{w^{\prime}\lambda}(\lambda). \label{B_w(lambda)}%
\end{equation}

\end{definition}

By writing $w=uv$ with $u\in W^{\lambda}$ and $v\in W_{\lambda}$, we get
$B_{w}(\lambda)=B_{u}(\lambda)$ from the characterization of the strong Bruhat
order recalled in \S \ref{subsec_KMA}.\ Thus we can and shall assume that both
$w$ and $w^{\prime}$ belong to $W^{\lambda}$ in (\ref{B_w(lambda)}). The
following Theorem has been established by Kashiwara and Littelmann.

\begin{theorem}
\label{Th_KL}Assume $\lambda$ is a dominant weight.

\begin{enumerate}
\item We have $s_{\lambda}^{w}=\sum_{b\in B_{w}(\lambda)}e^{\mathrm{wt}(b)}$.

\item For any reduced decomposition $s_{i_{1}}\cdots s_{i_{\ell}}$ of $w$, we
have $B_{w}(\lambda):=\{\tilde{f}_{i_{1}}^{k_{1}}\cdots\tilde{f}_{i_{\ell}%
}^{k_{\ell}}(b_{\lambda})\mid(k_{1},\ldots,k_{\ell})\in\mathbb{Z}_{\geq
0}^{\ell}\}.$
\end{enumerate}
\end{theorem}

It is also interesting to define%
\[
B_{w}(\infty)=\lim_{\lambda\rightarrow+\infty}B_{w}(\lambda):=\{\tilde
{f}_{i_{1}}^{k_{1}}\cdots\tilde{f}_{i_{\ell}}^{k_{\ell}}(b_{\emptyset}%
)\mid(k_{1},\ldots,k_{\ell})\in\mathbb{Z}_{\geq0}^{\ell}\}.
\]
Thus, from the above result, we deduce that to compute the Demazure crystal
$B_{w}(\lambda)$, it suffices to

\begin{itemize}
\item compute the Key map $K^{R}$ on $B(\lambda)$.

\item compute the strong Bruhat order on $W^{\lambda}$, or alternatively on
the vertices of $O(\lambda)$.
\end{itemize}

\section{Recursive computations of the Keys and the strong Bruhat order}

\label{Sec_Th_fund}We shall describe in this section procedures for computing
the Keys using combinatorial $\mathrm{R}$-matrices and the strong Bruhat order
on the orbit of the highest weight vertex.

\subsection{Keys and combinatorial $\mathrm{R}$-matrices}

\label{subsec_Fock}Consider $\lambda,\mu$ two dominant weights. Then
$B(\lambda)\otimes B(\mu)$ contains a unique connected component
$B_{\lambda,\mu}(\lambda+\mu)$ isomorphic to the abstract crystal
$B(\lambda+\mu)$ with highest weight vertex $b_{\lambda,\mu}=b_{\lambda
}\otimes b_{\mu}$. Moreover, the crystals $B(\lambda)\otimes B(\mu)$ and
$B(\mu)\otimes B(\lambda)$ are isomorphic. In general, there are fewer
isomorphisms from $B(\lambda)\otimes B(\mu)$ to $B(\mu)\otimes B(\lambda)$
(see \cite{KT} for the description of such an isomorphism in the Kac-Moody
case). Nevertheless, each such isomorphism sends $B_{\lambda,\mu}(\lambda
+\mu)$ on $B_{\mu,\lambda}(\lambda+\mu)$ for there is only one connected
component in $B(\lambda)\otimes B(\mu)$ and $B(\mu)\otimes B(\lambda)$ of
highest weight $\lambda+\mu$ that we call \emph{principal}. We shall write
$\mathrm{R}$ the unique isomorphism from $B_{\lambda,\mu}(\lambda+\mu)$ to
$B_{\mu,\lambda}(\lambda+\mu)$.\ Also recall that $O_{\lambda,\mu}(\lambda
+\mu)$ is the orbit of $b_{\lambda}\otimes b_{\mu}$ in $B(\lambda)\otimes
B(\mu)$ under the action of $W$.

Given two crystals $B_{1}$ and $B_{2}$ the flip $\mathrm{F}$ is the bijection
from $B_{1}\otimes B_{2}$ to $B_{2}\otimes B_{1}$ defined by $\mathrm{F}%
(u\otimes v)=v\otimes u$ for any $u\otimes v$ in $B_{1}\otimes B_{2}$. This is
not a crystal isomorphism in general.

The previous definitions of $B_{\lambda,\mu}(\mu+\lambda),$ $O_{\lambda,\mu
}(\lambda+\mu)$ etc. extend naturally to the case where a sequence
$\lambda^{(1)},\ldots,\lambda^{(m)}$ of dominant weights is considered (rather
than just two dominant weights).\ Let us start with the easy following lemma.

\begin{lemma}
\label{Lem_W_tens_prod}Consider $w\cdot b_{\lambda^{(1)}}\otimes\cdots\otimes
b_{\lambda^{(m)}}\in O_{\lambda^{(1)},\ldots,\lambda^{(m)}}(\lambda
^{(1)}+\cdots+\lambda^{(m)})$.\ Then we have
\[
w\cdot b_{\lambda^{(1)}}\otimes\cdots\otimes b_{\lambda^{(m)}}=b_{w\lambda
^{(1)}}\otimes\cdots\otimes b_{w\lambda^{(m)}}.
\]
Moreover, for any $i\in I$ and any $k=1,\ldots,m$ we have%
\[
\left\{
\begin{array}
[c]{c}%
\varepsilon_{i}(b_{w\cdot\lambda^{(k)}})=0\text{ when }\ell(s_{i}%
w)=\ell(w)+1,\\
\varphi_{i}(b_{w\cdot\lambda^{(k)}})=0\text{ when }\ell(s_{i}w)=\ell(w)-1.
\end{array}
\right.
\]

\end{lemma}

\begin{proof}
This follows from (\ref{i-chain_orbit}), the definition of the action of the
generators $s_{i}$ and an easy induction on the length of $w$.
\end{proof}

\bigskip

Now, consider $w\cdot b_{\lambda,\mu}\in O_{\lambda,\mu}(\lambda+\mu)\subset
B(\lambda)\otimes B(\mu)$.

\begin{lemma}
\label{Lemm_orbit-tens}Assume $w\in W$. Then, we have
\[
w\cdot b_{\lambda,\mu}=b_{p_{\lambda}(w)\lambda}\otimes b_{p_{\mu}(w)\mu}.
\]

\end{lemma}

\begin{proof}
We can assume that $w\in W^{\lambda+\mu}$ and set $w\cdot b_{\lambda,\mu
}=b_{1}\otimes b_{2}$.\ By Lemma \ref{Lem_W_tens_prod}, we know that $b_{1}\in
O_{\lambda}(b_{\lambda})$ and $b_{2}\in O_{\mu}(b_{\mu})$. Thus we can set
$w\cdot b_{\lambda,\mu}=b_{w_{L}\lambda}\otimes b_{w_{R}\lambda}$ with
$(w_{L},w_{R})\in W^{\lambda}\times W^{\mu}$. Moreover, if we fix a reduced
expression of $w=s_{i_{1}}\cdots s_{i_{m}}\in W^{\lambda+\mu}$, we get a
directed path $\pi_{(\lambda,\mu)}^{w}$ in $O_{\lambda,\mu}(\lambda+\mu)$ from
$b_{\lambda,\mu}$ to $b$ obtained by applying crystals operators $\tilde
{f}_{i},i\in I$. By using the tensor product rules (\ref{tens_crys}) for these
operators, this thus yields also a directed path $\pi_{\lambda}^{w_{L}}$ in
$O_{\lambda}(b_{\lambda})$ from $b_{\lambda}$ to $b_{w_{L}\lambda}$ and a
directed path $\pi_{\mu}^{w_{R}}$ in $O_{\mu}(b_{\mu})$ from $b_{\mu}$ to
$b_{w_{R}\mu}$.\ The equivalence between directed paths and elements in
$W^{\lambda}$ and $W^{\mu}$ (see \S \ref{subsub_reduc_strongBO}) then imposes
that we have $(w_{L},w_{R})=(p_{\lambda}(w),p_{\mu}(w))$.
\end{proof}

\bigskip

Consider $b=u\otimes v$ in $B_{\lambda,\mu}(\lambda+\mu)$ and set%
\[
K^{L}(b)=u^{L}\otimes v^{L},\quad K^{R}(b)=u^{R}\otimes v^{R}.
\]

\begin{lemma}
\label{Lem_K_tens}We have $u^{L}=K^{L}(u)$ and $v^{R}=K^{R}(v)$.
\end{lemma}

\begin{proof}
For $m$ with sufficiently many factors, we get by definition of $K^{L}(b)$ and
$K^{R}(b)$%
\begin{equation}
K_{m}(b)=K^{L}(b)\otimes\cdots\otimes K^{R}(b)\label{K_m(l+m)}%
\end{equation}
where $K_{m}$ is the crystal embedding from $B(b_{\lambda,\mu})$ in
$B(b_{\lambda,\mu}^{\otimes m})$ defined in (\ref{embdedd}). Now the crystals
$B(b_{\lambda,\mu}^{\otimes m})$ and $B(b_{\lambda}^{\otimes m}\otimes b_{\mu
}^{m})$ are isomorphic for their highest weight vertices $b_{\lambda,\mu
}^{\otimes m}=(b_{\lambda}\otimes b_{\mu})^{m}$ and $b_{\lambda}^{\otimes
m}\otimes b_{\mu}^{m}$ have the same highest weight $m(\lambda+\mu)$.$\ $The
isomorphism $\mathrm{I}$ from $B(b_{\lambda,\mu}^{\otimes m})$ and
$B(b_{\lambda}^{\otimes m}\otimes b_{\mu}^{m})$ and its converse
$\mathrm{I}^{-1}$ are obtained by composing $\mathrm{R}$-matrices whose
actions on the previous highest weight vertices reduce to the flip of
components $b_{\lambda}$ and $b_{\mu}$.\ In particular $\mathrm{I}$ and
$\mathrm{I}^{-1}$ fix the leftmost and rightmost components in the vertices of
$B(b_{\lambda,\mu}^{\otimes m})$ and $B(b_{\lambda}^{\otimes m}\otimes b_{\mu
}^{m})$. Since $m$ can be any integer with sufficiently many factors, one can
choose such a integer $m$ so that (\ref{K_m(l+m)}) holds and simultaneously%
\[
K_{m}^{\prime}(b)=K_{m}(u)\otimes K_{m}(v)=K^{L}(u)\otimes\cdots\otimes
K^{R}(u)\otimes K^{L}(v)\otimes\cdots\otimes K^{R}(v)
\]
where $K_{m}^{\prime}=\mathrm{I}\circ K_{m}$ by Lemma \ref{Lem_Kmtild}. Since
$\mathrm{I}^{-1}$ fixes the leftmost and rightmost components in
$K_{m}^{\prime}(b)$ we are done.
\end{proof}

\bigskip

Now we can show that the action of $W$ commutes with the flip $F$ on the orbit
$O_{\lambda,\mu}(\lambda+\mu)$.

\begin{proposition}
For any $w\in W$ and any vertex $b\in O_{\lambda,\mu}(\lambda+\mu)$, we have
\[
w\circ F(b)=F\circ w(b).
\]

\end{proposition}

\begin{proof}
Write $b=u(b_{\lambda}\otimes b_{\mu})$. One the one hand , Lemma
\ref{Lemm_orbit-tens} gives $w(b)=wu(b_{\lambda}\otimes b_{\mu})=b_{p_{\lambda
}(wu)\lambda}\otimes b_{p_{\mu}(wu)\mu}$ and thus $F\circ w(b)=b_{p_{\mu
}(wu)\mu}\otimes b_{p_{\lambda}(wu)\mu}$. On the other hand, we get $w\circ
F(b)=w\circ F(b_{p_{\lambda}(u)\lambda}\otimes b_{p_{\mu}(u)\mu})=w(b_{p_{\mu
}(u)\mu}\otimes b_{p_{\lambda}(u)\lambda})=wu(b_{\mu}\otimes b_{\lambda
})=b_{p_{\mu}(wu)\mu}\otimes b_{p_{\lambda}(wu)\lambda}$. Therefore, we have
$w\circ F(b)=F\circ w(b)$ as desired.
\end{proof}

\begin{corollary}
\label{Cor_R=F}The maps $\mathrm{R}$ and $\mathrm{F}$ coincide on
$O_{\lambda,\mu}(\lambda+\mu)$.
\end{corollary}

\begin{proof}
Consider $b=w\cdot b_{\lambda,\mu}$ in $O_{\lambda,\mu}(\lambda+\mu)$. On the
one hand side, we have
\[
\mathrm{R}(w\cdot b_{\lambda,\mu})=w\cdot\mathrm{R}(b_{\lambda,\mu})=w\cdot
b_{\mu,\lambda}%
\]
because $\mathrm{R}$ is a crystal isomorphism (and thus commutes with the
action of $W$) and $\mathrm{R}(b_{\lambda,\mu})=b_{\mu,\lambda}$. On the other
side we get%
\[
\mathrm{F}(w\cdot b_{\lambda,\mu})=w\cdot\mathrm{F}(b_{\lambda,\mu})=w\cdot
b_{\mu,\lambda}%
\]
by using the previous proposition and the equality $\mathrm{F}(b_{\lambda,\mu
})=b_{\mu,\lambda}$.
\end{proof}

\subsection{Reduction to smaller dominant weights}

Denote by $\prec$ the partial dominant order on $P_{+}$ such that $\mu
\preceq\lambda$ if and only if $\lambda-\mu\in P_{+}$ and resume the notation
of \S \ref{subsec_Fock}. Our aim is now to compute the left and right Keys of
any vertex in $B_{\lambda,\mu}(\lambda+\mu)$ as a tensor product of Keys in
$B(\lambda)$ and $B(\mu)$. For any $b=u\otimes v$ in $B_{\lambda,\mu}%
(\lambda+\mu),$ set $\mathrm{R}(b)=\widetilde{b}=\widetilde{v}\otimes
\widetilde{u}$ in $B_{\mu,\lambda}(\lambda+\mu)$.

\begin{theorem}
\label{Th_Fund}We have%
\[
K^{L}(b)=K^{L}(u)\otimes K^{L}(\widetilde{v})\text{ }\quad\text{and }\quad
K^{R}(b)=K^{R}(\widetilde{u})\otimes K^{R}(v).
\]

\end{theorem}

\begin{proof}
We prove the first equality, the arguments being similar for the second
one.\ Write $K^{L}(b)=u^{L}\otimes v^{L}$.\ By Lemma \ref{Lem_K_tens}, we
first get that $u^{L}=K^{L}(u)$.\ We also have $K^{L}(\mathrm{R}%
(b))=\mathrm{R}(K^{L}(b))$ because $\mathrm{R}$ is a crystal isomorphism as in
(\ref{K_comm_fi}). By Corollary \ref{Cor_R=F} and once again, Lemma
\ref{Lem_K_tens} we deduce the equality%
\[
K^{L}(\mathrm{R}(b))=K^{L}(\widetilde{v}\otimes\widetilde{u})=K^{L}%
(\widetilde{v})\otimes\widetilde{u}^{L}=\mathrm{R}(K^{L}(b))=\mathrm{R}%
(K^{L}(u)\otimes v^{L})=v^{L}\otimes K^{L}(u).
\]
Thus, $v^{L}=K^{L}(\widetilde{v})$ as desired.
\end{proof}

Now consider $\mathcal{S}=(\lambda^{(1)},\ldots,\lambda^{(l)})$ a sequence of
dominant weights and write $\lambda=\lambda^{(1)}+\cdots+\lambda^{(l)}$. Let
$B_{\mathcal{S}}(\lambda)$ be the unique connected component in $%
{\textstyle\bigotimes\limits_{k=1}^{l}}
B(\lambda^{(k)})$ of highest weight $\lambda$.\ Its highest weight vertex is
$b_{\mathcal{S}}=b_{\lambda^{(1)}}\otimes\cdots\otimes b_{\lambda^{(l)}}$. For
any $k=1,\ldots,l,$ denote by $\theta_{k}^{L}$ the unique crystal isomorphism
from $B_{\mathcal{S}}(\lambda)$ to $B_{\mathcal{S}_{L}^{(k)}}(\lambda)$ where
$\mathcal{S}_{L}^{(k)}=(\lambda^{(k)},\lambda^{(1)},\ldots,\lambda^{(l)})$ (in
particular $\mathcal{S}_{L}^{(1)}=\mathcal{S}$). Write similarly $\theta
_{k}^{R}$ the unique crystal isomorphism from $B_{\mathcal{S}}(\lambda)$ to
$B_{\mathcal{S}_{R}^{(k)}}(\lambda)$ where $\mathcal{S}_{R}^{(k)}%
=(\lambda^{(1)},\ldots,\lambda^{(l)},\lambda^{(k)})$ (in particular
$\mathcal{S}_{R}^{(l)}=\mathcal{S}$)

For any $b=b_{1}\otimes\cdots\otimes b_{l}$ in $B_{\mathcal{S}}(\lambda)$, set
$\theta_{k}^{L}(b)=b_{1}^{L}(k)\otimes\cdots\otimes b_{l}^{L}(k)$ and
$\theta_{k}^{R}(b)=b_{1}^{R}(k)\otimes\cdots\otimes b_{l}^{R}(k)$. In
particular $b_{1}^{L}(k)\in B(\lambda^{(k)})$ and $b_{l}^{R}(k)\in
B(\lambda^{(k)})$ for any $k=1,\ldots,l$. An easy induction yields the
following corollary of Theorem \ref{Th_Fund}.

\begin{corollary}
\label{Cor_fund}For any $b\in B_{\mathcal{S}}(\lambda),$ we have
\[
K^{L}(b)=%
{\textstyle\bigotimes\limits_{k=1}^{l}}
K^{L}(b_{1}^{L}(k))\text{ }\quad\text{and }\quad K^{R}(b)=%
{\textstyle\bigotimes\limits_{k=1}^{l}}
K^{R}(b_{l}^{R}(k)).
\]

\end{corollary}

\begin{remark}
The previous corollary reduces the computation of the Keys for a dominant
weight $\lambda$ to that of $\mathrm{R}$-matrices and Keys for dominant
weights less that $\lambda$ for the order $\prec$ on $P_{+}$. For finite types
and for affine type $A$, we shall see that this gives an efficient procedure
by decomposing $\lambda$ on the basis of fundamental weights.
\end{remark}

\subsection{Recursive computation of the strong Bruhat order}

We resume the notation of the previous \S \ of this section. In
\S \ref{subsub_reduc_strongBO}, we have also seen that the elements of
$W^{\lambda+\mu}$ are matched with the vertices of $O_{\lambda,\mu}%
(\lambda+\mu)$.

\begin{proposition}
\label{Prop_proj_tens_prod}Consider $w\cdot b_{\lambda,\mu}$ and $w^{\prime
}\cdot b_{\lambda,\mu}$ in $O_{\lambda,\mu}(\lambda+\mu)\subset B(\lambda
)\otimes B(\mu)$ with $w$ and $w^{\prime}$ in $W^{\lambda+\mu}$. Then
\[
w\cdot b_{\lambda,\mu}=b_{p_{\lambda}(w)\lambda}\otimes b_{p_{\mu}(w)\mu
}\text{ and }w^{\prime}\cdot b_{\lambda,\mu}=b_{p_{\lambda}(w^{\prime}%
)\lambda}\otimes b_{p_{\mu}(w^{\prime})\mu}.
\]
Moreover%
\[
w\trianglelefteq w^{\prime}\text{ if and only if }p_{\lambda}%
(w)\trianglelefteq p_{\lambda}(w^{\prime})\text{ and }p_{\mu}%
(w)\trianglelefteq p_{\mu}(w^{\prime}).
\]

\end{proposition}

\begin{proof}
The first statement of the proposition comes by applying Lemma
\ref{Lemm_orbit-tens} to $w$ and $w^{\prime}$ and the second one is Lemma
\ref{Lem_SBOCosets}.
\end{proof}

\bigskip

Now, consider $\mathcal{S}=(\lambda^{(1)},\ldots,\lambda^{(l)})$ a sequence of
dominant weights and write $O_{\mathcal{S}}(\lambda)$ for the orbit of
$b_{\mathcal{S}}$ in $B_{\mathcal{S}}(\lambda)$.

Proposition \ref{Prop_proj_tens_prod} and an easy induction yields the
following corollary which permits a recursive computation of the strong Bruhat
order. It will be of particular interest in the following sections when the
$\lambda^{(k)},k=1,\ldots,l$ are fundamental weights and the crystal
$B(\lambda)$ has a convenient realization in terms of tableaux or abaci.

\begin{corollary}
\label{Cor_inclus_columns} \ 

\begin{enumerate}
\item For any $b_{1}\otimes\cdots\otimes b_{l}=w\cdot b_{\mathcal{S}}\in
O_{\mathcal{S}}(\lambda)$ with $w\in W^{\lambda^{(1)}+\cdots+\lambda^{(l)}}$,
there exists a unique $l$-tuple $(w_{1},\ldots,w_{l})\in\prod_{k=1}%
^{l}W^{\lambda^{(k)}}$ such that $b_{k}=w_{k}\cdot b_{\lambda^{(k)}}$ for any
$k=1,\ldots,l$.

\item We have $w_{k}=p_{\lambda^{(k)}}(w)$.

\item Given $w\cdot b_{\mathcal{S}}\in O_{\mathcal{S}}(\lambda)$ and
$w^{\prime}\cdot b_{\mathcal{S}}\in O_{\mathcal{S}}(\lambda)$ with
$w,w^{\prime}\in W^{\lambda^{(1)}+\cdots+\lambda^{(l)}}$, we have
\[
w\trianglelefteq w^{\prime}\text{ if and only if }w_{k}\trianglelefteq
w_{k}^{\prime}\text{ for any }k=1,\ldots,l.
\]

\end{enumerate}
\end{corollary}

\section{Determination of the Demazure crystals by Keys in finite types}

By Theorems \ref{Th_KL} and \ref{Th_Fund}, given any dominant weight $\lambda$
expressed as a sum of fundamental weights, we can conveniently compute the
Demazure crystals $B_{w}(\lambda)$ as soon as we have efficient procedures for

\begin{itemize}
\item computing the combinatorial $\mathrm{R}$-matrix (or at least its
restriction to the principal connected components) on tensor product of
fundamental crystals (i.e. crystal with fundamental highest weights),

\item computing the Key for fundamental crystals,

\item computing the strong Bruhat order on $W^{\lambda}.$
\end{itemize}

\subsection{The finite type A}

\label{ftyA}

\label{Subsec_FinitetypeA}We start by recalling the results of Lascoux and
Sch\"{u}tzenberger \cite{Las1}.\ In type $A_{n}$, the crystal $B(\omega_{i}),$
$i=1,\ldots,n$ is conveniently realized as the set of columns of height $i$ on
$\{1<\cdots<n<n+1\}$.\ Then the dominant weight $\omega_{i}$ is minuscule
which implies that $K^{L}(C)=K^{R}(C)$ for any column $C\in B(\omega_{i}%
)$.\ Also the combinatorial \textrm{$R$}-matrices can be computed by using the
Jeu de Taquin procedure or the insertion scheme on semistandard tableaux. More
generally given a sequence $\mathcal{S}=(\omega_{i_{1}},\ldots,\omega_{i_{l}%
})$ of dominant weights such that $i_{1}\geq\cdots\geq i_{l}$, the vertices of
the crystal $B_{\mathcal{S}}(\lambda)$ with $\lambda=\omega_{i_{1}}%
+\cdots+\omega_{i_{l}}$ defined in \S \ \ref{subsec_Fock} can be identified
with the semistandard tableaux of shape $\lambda$ (see \cite{kash} and the
example below). The highest weight tableau is the tableau $T(\lambda)$ of
shape $\lambda$ with entries $i$ in row $i$ for any $i=1,\ldots,n$.\ The
elements of the orbit $O_{\mathcal{S}}(\lambda)$ of $T(\lambda)$ are the
semistandard tableaux $T=C_{1}\otimes\cdots\otimes C_{l}$ of shape $\lambda$
verifying the chain of inclusions $C_{l}\subset\cdots\subset C_{2}\subset
C_{1}$.\ Also for two such tableaux $T=C_{1}\otimes\cdots\otimes C_{l}$ and
$T^{\prime}=C_{1}^{\prime}\otimes\cdots\otimes C_{l}^{\prime}$ with $T=w\cdot
T(\lambda)$ and $T=w^{\prime}\cdot T(\lambda)$ and $(w,w^{\prime}%
)\in(W^{\lambda})^{2}$, we have $w\trianglelefteq w^{\prime}$ if and only if
$C_{k}C_{k}^{\prime}$ is a semistandard tableau for any $k=1,\ldots,l.$ This
is a direct consequence of \ref{Cor_inclus_columns}.\ Equivalently one gets
that the Strong Bruhat order on $W^{\lambda}$ is just the product of the
strong Bruhat orders on the cosets $W^{\omega_{i}},i=1,\ldots,l$.

\begin{example}
Let us compute the Key of the tableau%
\[
T=%
\begin{tabular}
[c]{|l|l|l}\hline
$1$ & $2$ & \multicolumn{1}{|l|}{$2$}\\\hline
$3$ & $4$ & \multicolumn{1}{|l|}{$4$}\\\hline
$4$ & $5$ & \\\cline{1-2}%
\end{tabular}
\
\]
corresponding to $v=(3,3,2)$.\ By using the Jeu de Taquin procedure, we get
for the associated generalized tableaux of shape $(3,2,3)$ and $(2,3,3)$%
\[%
\begin{tabular}
[c]{|l|l|l|}\hline
$1$ & $2$ & $2$\\\hline
$3$ & $4$ & $4$\\\hline
$4$ &  & $5$\\\cline{1-1}\cline{3-3}%
\end{tabular}
\ \text{ and }%
\begin{tabular}
[c]{l|l|l|}\hline
\multicolumn{1}{|l|}{$1$} & $2$ & $2$\\\hline
\multicolumn{1}{|l|}{$4$} & $3$ & $4$\\\hline
& $4$ & $5$\\\cline{2-3}%
\end{tabular}
\
\]
which gives
\[
T^{L}=%
\begin{tabular}
[c]{|l|l|l}\hline
$1$ & $1$ & \multicolumn{1}{|l|}{$1$}\\\hline
$3$ & $3$ & \multicolumn{1}{|l|}{$4$}\\\hline
$4$ & $4$ & \\\cline{1-2}%
\end{tabular}
\text{ and }T^{R}=%
\begin{tabular}
[c]{|l|l|l}\hline
$2$ & $2$ & \multicolumn{1}{|l|}{$2$}\\\hline
$4$ & $4$ & \multicolumn{1}{|l|}{$4$}\\\hline
$5$ & $5$ & \\\cline{1-2}%
\end{tabular}
\ .
\]

\end{example}

\subsection{Other finite types}

\subsubsection{Classical types}

Thanks to Corollary \ref{Cor_fund} the computation of the Keys in types
$B_{n},C_{n}$ and $D_{n}$ becomes very closed to that in type $A_{n}$. We
shall describe it for type $C_{n}$ and let to the reader its adaptation to
types $B_{n}$ and $D_{n}$. For a review on the combinatorics of crystals in
classical types we refer to \cite{LecSurv}.\ There exists a convenient notion
of symplectic tableaux compatible with crystal basis theory \cite{KN}.\ In
type $C_{n}$, the dominant weights can be identified with partitions exactly
as in type $A_{n}$.\ 

A tableau $T=C_{1}\cdots C_{l}$ of type $C_{n}$ and shape a partition
$\lambda$ is a filling of the Young diagram $\lambda$ by letters of
$\{1<\cdots<n<\overline{n}<\cdots<\overline{1}\}$ such that each column
$C_{l}$ is admissible and the split form of $T$ is semistandard.\ A column $C$
is admissible when it can be split in a pair $(lC,rC)$ of columns contained no
pair of letters $(z,\overline{z})$ with $z\in\{1,\ldots,n\}$ by using the
following procedure. Let $I=\{z_{1}>\cdot\cdot\cdot>z_{r}\}$ the set of
unbarred letters $z$ such that the pair $(z,\overline{z})$ occurs in $C$. The
column $C$ can be split when there exists (see the example below) a set
$J=\{t_{1}>\cdot\cdot\cdot>t_{r}\}\subset\{1,\ldots n\}$ such that:

\begin{itemize}
\item $t_{1}$ is the greatest letter of $\{1,\ldots n\}$ satisfying:
$t_{1}<z_{1},t_{1}\notin C$ and $\overline{t_{1}}\notin C,$

\item for $i=2,...,r$, $t_{i}$ is the greatest letter of $\{1,\ldots n\}$
satisfying: $t_{i}<\min(t_{i-1,}z_{i}),$ $t_{i}\notin C$ and $\overline{t_{i}%
}\notin C.$
\end{itemize}

\noindent In this case write:

\noindent$\mathrm{r}C$ for the column obtained by changing in $C,$
$\overline{z}_{i}$ into $\overline{t}_{i}$ for each letter $z_{i}\in I$ and by
reordering if necessary,

\noindent$\mathrm{l}C$ for the column obtained by changing in $C,$ $z_{i}$
into $t_{i}$ for each letter $z_{i}\in I$ and by reordering if necessary.

\noindent Admissible columns with $i$ boxes label the vertices of
$B(\omega_{i})$. Moreover for any $C$ in $B(\omega_{i})$, we have
$C^{L}=\mathrm{l}C$ and $C^{R}=\mathrm{r}C$, that is the previous procedure
give the left and right Keys of a column.

\noindent Now $T$ is a tableau of type $C_{n}$ when its split form
$\mathrm{spl}(T)=\mathrm{l}C_{1}\mathrm{r}C_{1}\cdots\mathrm{l}C_{l}%
\mathrm{r}C_{l}$ is semistandard.

As in type $A$, on associates to the sequence $\mathcal{S}=(\omega_{i_{1}%
},\ldots,\omega_{i_{l}})$ of dominant weights such that $i_{1}\geq\cdots\geq
i_{l}$ the crystal $B_{\mathcal{S}}(\lambda)$ with $\lambda=\omega_{i_{1}%
}+\cdots+\omega_{i_{l}}$.\ Its vertices then coincide with the tableaux of
type $C_{n}$ and shape $\lambda$. The $\mathrm{R}$-matrix $B_{(\omega
_{i},\omega_{j})}(\omega_{i}+\omega_{j})\rightarrow B_{(\omega_{j},\omega
_{i})}(\omega_{i}+\omega_{j})$ can be computed by using Sheats symplectic Jeu
de Taquin (which does not coincide with the restriction of the usual Jeu de
Taquin on symplectic tableaux) or the bumping procedure on symplectic
tableaux. The vertices in $O_{\mathcal{S}}(\lambda)$ are the tableaux of type
$C_{n}$ of the form $T=C_{1}\cdots C_{l}$ where $C_{l}\subset\cdots\subset
C_{1}$ and no pair of letters $(z,\overline{z})$ in each column $C_{k}$. As in
type $A$, for $T=w\cdot T(\lambda)$ and $T=w^{\prime}\cdot T(\lambda)$ and
$(w,w^{\prime})\in(W^{\lambda})^{2}$, we have $w\trianglelefteq w^{\prime}$ if
and only if $C_{k}C_{k}^{\prime}$ is a semistandard tableau for any
$k=1,\ldots,l.$

\begin{example}
Let us assume $n=4$ and compute the right Key of the tableau%
\[
T=%
\begin{tabular}
[c]{|l|l}\hline
$\mathtt{1}$ & \multicolumn{1}{|l|}{$\mathtt{2}$}\\\hline
$\mathtt{2}$ & \multicolumn{1}{|l|}{$\mathtt{4}$}\\\hline
$\mathtt{4}$ & \multicolumn{1}{|l|}{$\mathtt{\bar{4}}$}\\\hline
$\mathtt{\bar{4}}$ & \\\cline{1-1}%
\end{tabular}
\ \text{ with }\mathrm{spl}(T)=%
\begin{tabular}
[c]{|l|l|ll}\hline
$\mathtt{1}$ & $\mathtt{1}$ & $\mathtt{2}$ & \multicolumn{1}{|l|}{$\mathtt{2}%
$}\\\hline
$\mathtt{2}$ & $\mathtt{2}$ & $\mathtt{3}$ & \multicolumn{1}{|l|}{$\mathtt{4}%
$}\\\hline
$\mathtt{3}$ & $\mathtt{4}$ & $\mathtt{\bar{4}}$ &
\multicolumn{1}{|l|}{$\mathtt{\bar{3}}$}\\\hline
$\mathtt{\bar{4}}$ & $\mathtt{\bar{3}}$ &  & \\\cline{1-2}%
\end{tabular}
\
\]
We have
\[
R_{(\omega_{4},\omega_{3})}(T)=%
\begin{tabular}
[c]{l|l|}\hline
\multicolumn{1}{|l|}{$\mathtt{1}$} & $\mathtt{2}$\\\hline
\multicolumn{1}{|l|}{$\mathtt{2}$} & $\mathtt{4}$\\\hline
\multicolumn{1}{|l|}{$\mathtt{3}$} & $\mathtt{\bar{4}}$\\\hline
& $\mathtt{\bar{3}}$\\\cline{2-2}%
\end{tabular}
\ \text{ with }\mathrm{spl}(R_{(\omega_{4},\omega_{3})}(T))=%
\begin{tabular}
[c]{ll|l|l|}\hline
\multicolumn{1}{|l}{$\mathtt{1}$} & \multicolumn{1}{|l|}{$\mathtt{1}$} &
$\mathtt{1}$ & $\mathtt{2}$\\\hline
\multicolumn{1}{|l}{$\mathtt{2}$} & \multicolumn{1}{|l|}{$\mathtt{2}$} &
$\mathtt{2}$ & $\mathtt{4}$\\\hline
\multicolumn{1}{|l}{$\mathtt{3}$} & \multicolumn{1}{|l|}{$\mathtt{3}$} &
$\mathtt{\bar{4}}$ & $\mathtt{\bar{3}}$\\\hline
&  & $\mathtt{\bar{3}}$ & $\mathtt{\bar{1}}$\\\cline{3-4}%
\end{tabular}
\
\]
which gives
\[
T^{L}=%
\begin{tabular}
[c]{|l|l}\hline
$\mathtt{1}$ & \multicolumn{1}{|l|}{$\mathtt{1}$}\\\hline
$\mathtt{2}$ & \multicolumn{1}{|l|}{$\mathtt{2}$}\\\hline
$\mathtt{3}$ & \multicolumn{1}{|l|}{$\mathtt{3}$}\\\hline
$\mathtt{\bar{4}}$ & \\\cline{1-1}%
\end{tabular}
\ \text{ and }T^{R}=%
\begin{tabular}
[c]{|l|l}\hline
$\mathtt{2}$ & \multicolumn{1}{|l|}{$\mathtt{2}$}\\\hline
$\mathtt{4}$ & \multicolumn{1}{|l|}{$\mathtt{4}$}\\\hline
$\mathtt{\bar{3}}$ & \multicolumn{1}{|l|}{$\mathtt{\bar{3}}$}\\\hline
$\mathtt{\bar{1}}$ & \\\cline{1-1}%
\end{tabular}
\
\]

\end{example}

\subsubsection{Exceptional types}

For exceptional types, the Key in fundamental crystals can yet be computed
from the dilatation maps $K_{m}$ defined in (\ref{embdedd}) with $m\leq4$ (see
\cite{Lit2}). There is also relevant notions of tableaux (see \cite{BS} and
the references therein). Nevertheless, the combinatorial \textrm{$R$}-matrices
for fundamental crystals, the orbit of the highest weight vertex in the
crystals and the strong Bruhat order on this orbit become more complex to
compute beyond type $G_{2}$ (for which the model remains simple and there is a
bumping algorithm (see \cite{LecSurv})).

\section{Determination of the Demazure crystals by Keys in affine type A}

\label{Sect_affineA}In this section we assume $\mathfrak{g}=\widehat
{\mathfrak{sl}}_{e}$ is the affine Lie algebra of type $A_{e-1}^{(1)}$.\ A
sequence $\mathbf{s}=(s_{1},\ldots,s_{l})\in\mathbb{Z}^{l}$ is called a
multicharge. It defines the dominant weight $\Lambda_{\mathbf{s}}=\sum
_{i=1}^{l}\omega_{s_{i}\operatorname{mod}e}$ of level $l$ where $\omega
_{0},\ldots,\omega_{e-1}$ are the fundamental weights of $\widehat
{\mathfrak{sl}}_{e}$.

\subsection{The level 1}

We now review a convenient realization of the crystals $B(\omega
_{i}),i=0,\ldots,e-1$ by abaci.\ Recall that a partition is a nonincreasing
sequence $\lambda=(\lambda_{1}\geq\cdots\geq\lambda_{m})$ of nonnegative
integers.\ One can assume this sequence is infinite by adding parts equal to
zero. Each partition will be identified with its Young diagram. To each box
(also called node) $b$ of a partition $\lambda$, one associates its content
$c(b)=v-u$ where $u$ and $v$ are such that $b$ belongs to the $u$-th row and
the $v$-th column of $\lambda$, respectively. A partition is completely
determined by its beta numbers. These are the contents of its extended rim
obtained by adding one box to the right end of each row.\ The removable nodes
of $\lambda$ are the nodes located at the ends of its rows which yet yield a
partition when they are removed from $\lambda$. The addable nodes of $\lambda$
are the nodes in its extended rim which yield a partition when they are added
to $\lambda$.

Now fix $s\in\mathbb{Z}$.\ The symbol of $\lambda,$ denoted by $S_{s}%
(\lambda)$ is the list of its beta numbers translated by $s$.\ Alternatively,
one can consider the abacus $L_{s}(\lambda)$ which is obtained by decorating
$\mathbb{Z}$ with black and white beads such that the black beads corresponds
to the integers in $S_{s}(\lambda)$. Since $\lambda$ is assumed to have an
infinite number of zero parts, both $S_{s}(\lambda)$ and $L_{s}(\lambda)$ are
infinite. Nevertheless, only the nonzero parts of $\lambda$ are relevant which
is easy to make apparent when $S_{s}(\lambda)$ and $L_{s}(\lambda)$ are
pictured (see the following example). An addable node in $\lambda$ corresponds
in $L_{s}(\lambda)$ to a black bead with a white bead at its right whereas a
removable node corresponds to a black bead with a white bead at its left.

The set of symbols can be endowed with the structure of a type $A_{e-1}^{(1)}%
$-crystal.\ For any $i\in\{0,\ldots e-1\}$, the $i$-nodes of $\lambda$ are
those of content $x=i\operatorname{mod}e$. Let $w_{i}$ be the word on the
alphabet $\{A,R\}$ obtained by reading from right to left the entries $x$ of
$S_{s}(\lambda)$ such that $x=i\operatorname{mod}e$ or
$x=i+1\operatorname{mod}e$ corresponding to addable or removable $i$-nodes in
$\lambda$.\ We shall say that $w_{i}$ is the $\{A,R\}$-word of $\lambda
$.\ Delete recursively each factor $RA$ until obtain a reduced word of the
form $\widetilde{w}_{i}=A^{a}R^{r}$.\ Then $\tilde{f}_{i}(S_{s}(\lambda))$ is
obtained by changing in $S_{s}(\lambda)$ the rightmost entry $x$ appearing in
$\widetilde{w}_{i}$ into $x+1$ if $a>0$ and is zero otherwise. Is is easy to
check that $S_{s}(\emptyset)$ is then a source vertex. In fact the connected
component $B_{s}(\emptyset)$ of $S_{s}(\emptyset)$ is isomorphic to the
$\widehat{\mathfrak{sl}}_{e}$-crystal $B(\omega_{s\operatorname{mod}e})$ (see
for example \cite{GJ} Chapter 6 and the references therein). Also one can
prove that the vertices in $B_{s}(\emptyset)$ are the symbols $S_{s}(\lambda)$
corresponding to $e$-regular partitions, that is to partitions with no part
repeated strictly more than $e-1$ times. Alternatively, $\lambda$ is
$e$-regular if there is no sequence of $e$ black beads in $L_{s}(\lambda)$.

Observe that when $e$ tends to infinity, the previous construction yields the
crystal $B_{s}^{\infty}(S_{s}(\emptyset))$ which is isomorphic to the
$\mathfrak{sl}_{\infty}$-crystal with highest weight the $s$-th fundamental
weight. Also, up to rotation, the symbol $S_{s}(\lambda)$ is nothing but the
half-infinite column semistandard tableau on $\mathbb{Z}$ which is the natural
type $A_{\infty}$-extension of the finite columns used in
\S \ref{Subsec_FinitetypeA}.

\begin{example}
Consider the $3$-regular partition $\lambda=(5,3,3,2)$.\ Its beta numbers are
easily deduced from its Young diagram%
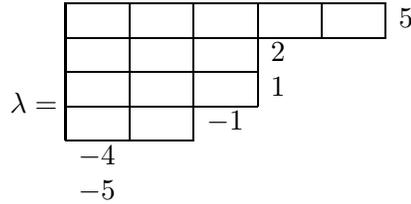
\[
\lambda=%
\begin{tabular}
[c]{p{0.5cm}p{0.5cm}p{0.5cm}p{0.5cm}p{0.5cm}p{0.5cm}p{0.5cm}}\cline{1-5}%
\multicolumn{1}{|p{0.5cm}}{} & \multicolumn{1}{|p{0.5cm}}{} &
\multicolumn{1}{|p{0.5cm}}{} & \multicolumn{1}{|p{0.5cm}}{} &
\multicolumn{1}{|p{0.5cm}}{} & \multicolumn{1}{|p{0.5cm}}{$5$} &
\\\cline{1-3}\cline{1-5}%
\multicolumn{1}{|p{0.5cm}}{} & \multicolumn{1}{|p{0.5cm}}{} &
\multicolumn{1}{|p{0.5cm}}{} & \multicolumn{1}{|p{0.5cm}}{$2$} &  &  &
\\\cline{1-3}%
\multicolumn{1}{|p{0.5cm}}{} & \multicolumn{1}{|p{0.5cm}}{} &
\multicolumn{1}{|p{0.5cm}}{} & \multicolumn{1}{|p{0.5cm}}{$1$} &  &  &
\\\cline{1-2}\cline{1-3}%
\multicolumn{1}{|p{0.5cm}}{} & \multicolumn{1}{|p{0.5cm}}{} &
\multicolumn{1}{|p{0.5cm}}{$-1$} &  &  &  & \\\cline{1-2}%
$-4$ &  &  &  &  &  & \\
$-5$ &  &  &  &  &  &
\end{tabular}
\
\]
and we get%
\[
S_{0}(\lambda)=%
\begin{tabular}
[c]{p{0.5cm}p{0.5cm}p{0.5cm}p{0.5cm}p{0.5cm}p{0.5cm}p{0.5cm}}%
$\cdots$ & $-5$ & $-4$ & $-1$ & $1$ & $2$ & $5$%
\end{tabular}
\]
The abacus $L_{0}(\lambda)$ is :%

\[
\begin{tikzpicture}[scale=0.5, bb/.style={draw,circle,fill,minimum size=2.5mm,inner sep=0pt,outer sep=0pt}, wb/.style={draw,circle,fill=white,minimum size=2.5mm,inner sep=0pt,outer sep=0pt}]
\node [wb] at (11,2) {};
\node [wb] at (10,2) {};
\node [wb] at (9,2) {};
\node [wb] at (8,2) {};
\node [wb] at (7,2) {};
\node [wb] at (6,2) {};
\node [bb] at (5,2) {};
\node [wb] at (4,2) {};
\node [wb] at (3,2) {};
\node [bb] at (2,2) {};
\node [bb] at (1,2) {};
\node [wb] at (0,2) {};
\node [bb] at (-1,2) {};
\node [wb] at (-2,2) {};
\node [wb] at (-3,2) {};
\node [bb] at (-4,2) {};
\node [bb] at (-5,2) {};
\node [bb] at (-6,2) {};
\node [bb] at (-7,2) {};
\node [bb] at (-8,2) {};
\node [bb] at (-9,2) {};
\draw[dashed](0.5,1.5)--node[]{}(0.5,2.5);
\end{tikzpicture}
\]



\end{example}

Recall that the hook length $h(b)$ of a node $b$ in the partition $\lambda$
(i.e. a box in its Young diagram) is the number of nodes located to the right
or below $b$ (weakly speaking, thus $b$ contributes to $h(b)$). A partition
$\lambda$ is called a $e$-core if it does not contain a node with hook length
$e$. There are alternative characterizations of the $e$-core (see for example
\cite{Las2}) of the partition $\lambda$.

\begin{proposition}
\label{Prop_CharCore}The following assertions are equivalent:

\begin{enumerate}
\item $\lambda$ is a $e$-core,

\item $\lambda$ does not contains any node with hook length $e$,

\item for any $i=0,\ldots,e-1,$ $w_{i}$ contains only nodes $A$ or only nodes
$R,$

\item for any $x$ in $S_{e}(\lambda)$, $x-e$ also belongs to $S_{e}(\lambda)$,

\item we have $L_{t}(\lambda)\subset L_{t+e}$ for any $t\in\mathbb{Z}$.
\end{enumerate}
\end{proposition}

Given two partitions $\lambda$ and $\mu$, we write $\lambda\subseteq\mu$ when
the Young diagram of $\lambda$ is contained in that of $\mu$. This defines the
inclusion order on partitions.

\begin{corollary}
\label{Cor_core}The orbit $O_{s}(\emptyset)$ of $\emptyset$ in $B_{s}%
(\emptyset)$ under the action of the Weyl group $W$ contains exactly the
$e$-cores. Moreover, under this correspondence, the strong Bruhat order on
$W^{\omega_{s}}$ coincides with the inclusion order on partitions.
\end{corollary}

In the following paragraph, we will see how generalize these two last results
in highest level. Now let us recall a combinatorial procedure described in
\cite{A} yielding the right Key $K_{s}^{R}(\lambda)$ of $S_{s}(\lambda)$ in
$B_{s}(\emptyset)$.\footnote{There is a similar procedure for computing the
left key $K_{s}^{L}(\lambda)$ also described in \cite{A}. Thus our forecoming
results can also be used to compute the left Key in arbitrary level.} First
set $U(S(\lambda))=\{x\in S(\lambda)\mid x-e\notin S(\lambda)\}$. Then
$K_{s}^{R}(\lambda)$ can be computed by the following algorithm:

\begin{enumerate}
\item If $U(S(\lambda))=\emptyset$, then $K_{s}^{R}(\lambda)=S_{s}(\lambda)$

\item Else let $p=\max\{x\in S(\lambda)\mid x-e\notin S(\lambda)\}$ and
$q=\min\{x>p\mid x\notin S(\lambda),x-e\in S(\lambda),x\neq
p\operatorname{mod}e\}$.\ Replace $S(\lambda)$ by $S(\lambda)\backslash
\{p\}\cup\{q\}$ and return to step 1.
\end{enumerate}

Observe the algorithm is well-defined for the set $\{x>p\mid x\notin
S(\lambda),x-e\in S(\lambda),x\neq p\operatorname{mod}e\}$ is not empty. Also
it terminates since the cardinality of $U(S(\lambda))$ decreases after
sufficiently iterations.

\begin{example}
Assume $e=3$ and
\[
S_{0}(\lambda)=%
\begin{tabular}
[c]{p{0.5cm}p{0.5cm}p{0.5cm}p{0.5cm}p{0.5cm}p{0.5cm}p{0.5cm}}%
$\cdots$ & $-5$ & $-4$ & $-1$ & $1$ & $2$ & $5$%
\end{tabular}
\ \
\]
Then we get $p=1$ and $q=8$.\ Only one iteration is needed and this gives%
\[
K_{s}^{R}(\lambda)=%
\begin{tabular}
[c]{p{0.5cm}p{0.5cm}p{0.5cm}p{0.5cm}p{0.5cm}p{0.5cm}p{0.5cm}}%
$\cdots$ & $-5$ & $-4$ & $-1$ & $2$ & $5$ & $8$%
\end{tabular}
\ \ \ \
\]
which is the symbol of the $3$-core $\mu=(8,6,4,2)$.
\end{example}

\subsection{Higher level}

\subsubsection{Uglov realization}

Let $\mathbf{s}=(s_{1},\ldots,s_{l})\in\mathbb{Z}^{l}$ be an arbitrary
multicharge. We now recall Uglov's realization of the crystal $B(\Lambda
_{\mathbf{s}})$. Its is quite similar to the level 1 case except one has to
consider $l$-partitions $\boldsymbol{\lambda}=(\lambda^{1},\ldots,\lambda
^{l})$ (i.e. sequences of partitions of length $l$) instead of partitions. To
the $l$-partition $\boldsymbol{\lambda}$ is associated its symbol which is the
sequence $S_{\boldsymbol{s}}(\boldsymbol{\lambda})=(S_{s_{l}}(\lambda
^{l}),\ldots,S_{s_{1}}(\lambda^{1}))$ of the the symbols associated to each
pair $(s_{k},\lambda^{k})$. The abacus $L_{\boldsymbol{s}}(\boldsymbol{\lambda
})=(L_{s_{l}}(\lambda^{l}),\ldots,L_{s_{1}}(\lambda^{1}))$ is defined similarly.

\begin{example}
The abacus of the $3$-partition $(1.1,2.2.2.1,9)$ with $\mathbf{s}=(4,6,1)$ is
\end{example}

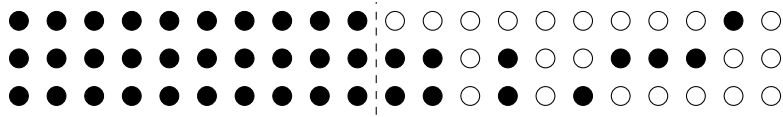
\begin{center}
\begin{tikzpicture}[scale=0.5, bb/.style={draw,circle,fill,minimum size=2.5mm,inner sep=0pt,outer sep=0pt}, wb/.style={draw,circle,fill=white,minimum size=2.5mm,inner sep=0pt,outer sep=0pt}]

	\node [wb] at (11,2) {};
	\node [bb] at (10,2) {};
	\node [wb] at (9,2) {};
	\node [wb] at (8,2) {};
	\node [wb] at (7,2) {};
	\node [wb] at (6,2) {};
	\node [wb] at (5,2) {};
	\node [wb] at (4,2) {};
	\node [wb] at (3,2) {};
	\node [wb] at (2,2) {};
	\node [wb] at (1,2) {};
	\node [bb] at (0,2) {};
	\node [bb] at (-1,2) {};
	\node [bb] at (-2,2) {};
	\node [bb] at (-3,2) {};
	\node [bb] at (-4,2) {};
	\node [bb] at (-5,2) {};
	\node [bb] at (-6,2) {};
	\node [bb] at (-7,2) {};
	\node [bb] at (-8,2) {};
	\node [bb] at (-9,2) {};
	
	\node [wb] at (11,1) {};
	\node [wb] at (10,1) {};
	\node [bb] at (9,1) {};
	\node [bb] at (8,1) {};
	\node [bb] at (7,1) {};
	\node [wb] at (6,1) {};
	\node [wb] at (5,1) {};
	\node [bb] at (4,1) {};
	\node [wb] at (3,1) {};
	\node [bb] at (2,1) {};
	\node [bb] at (1,1) {};
	\node [bb] at (0,1) {};
	\node [bb] at (-1,1) {};
	\node [bb] at (-2,1) {};
	\node [bb] at (-3,1) {};
	\node [bb] at (-4,1) {};
	\node [bb] at (-5,1) {};
	\node [bb] at (-6,1) {};
	\node [bb] at (-7,1) {};
	\node [bb] at (-8,1) {};
	\node [bb] at (-9,1) {};
	
	\node [wb] at (11,0) {};
	\node [wb] at (10,0) {};
	\node [wb] at (9,0) {};
	\node [wb] at (8,0) {};
	\node [wb] at (7,0) {};
	\node [bb] at (6,0) {};
	\node [wb] at (5,0) {};
	\node [bb] at (4,0) {};
	\node [wb] at (3,0) {};
	\node [bb] at (2,0) {};
	\node [bb] at (1,0) {};
	\node [bb] at (0,0) {};
	\node [bb] at (-1,0) {};
	\node [bb] at (-2,0) {};
	\node [bb] at (-3,0) {};
	\node [bb] at (-4,0) {};
	\node [bb] at (-5,0) {};
	\node [bb] at (-6,0) {};
	\node [bb] at (-7,0) {};
	\node [bb] at (-8,0) {};
	\node [bb] at (-9,0) {};

	\draw[dashed](0.5,-0.5)--node[]{}(0.5,2.5);
	
	\end{tikzpicture}


\end{center}

The set of symbols so obtained is also endowed with the structures of
$\widehat{\mathfrak{sl}}_{\infty}$ and $\widehat{\mathfrak{sl}}_{e}$-crystals
of level $l$. Nevertheless the $\widehat{\mathfrak{sl}}_{e}$-crystal structure
is not a tensor product of level $1$ crystals when $l>1$. Thus, we cannot
apply directly to the results of Section \ref{Sec_Th_fund}.\ We shall see in
\S \ \ref{SubsecKles} that there is another (closed) construction of level $l$
$\widehat{\mathfrak{sl}}_{e}$-crystals (called the Kleshchev realization)
which is by definition a tensor product of level $1$ affine crystals. We shall
consider both in the sequel notably because Uglov's version is easier to
connect to the combinatorics of non affine type $A$ and the two versions are
of common use in the literature.

\bigskip

First of all, to get the $\widehat{\mathfrak{sl}}_{\infty}$-structure,
consider $j\in\mathbb{Z}$ and $W_{j}$ the word on the alphabet $\{A,R\}$
obtained by reading from right to left and successively in $L_{s_{l}}%
(\lambda^{l}),\ldots,L_{s_{1}}(\lambda^{1})$, the entries $j$ or $j+1$
corresponding to addable or removable nodes.\ Delete recursively each factor
$RA$ in $W_{j}$ until get a reduced word of the form $\widetilde{W}_{j}%
=A^{a}R^{r}$.\ Then $\widetilde{F}_{j}(S_{s}(\lambda))$ is obtained by
changing in $S_{\boldsymbol{s}}(\boldsymbol{\lambda})$ the rightmost $j$
appearing in $\widetilde{W}_{j}$ into $j+1$ if $a>0$ and is zero otherwise. It
is easy to check that $S_{s}(\boldsymbol{\emptyset})$ is then a source vertex
of highest weight $\Lambda_{\mathbf{s}}^{\infty},$ thus its associated
connected component $B^{\infty}(S_{\boldsymbol{s}}(\boldsymbol{\emptyset}))$
is isomorphic to $B(\Lambda_{\mathbf{s}}^{\infty})$.\ 

\noindent Now, to define the $\widehat{\mathfrak{sl}}_{e}$-structure, consider
$i\in\{0,\ldots e-1\}$ and the word%
\begin{equation}
w_{i}=\prod_{p=-\infty}^{+\infty}W_{i+pe}. \label{w_idev}%
\end{equation}
Define $\widetilde{w}_{i}=A^{a}R^{r}$ from $w_{i}$ as previously by recursive
deletion of the factors $RA$. The nodes surviving in $\widetilde{w}_{i}$ are
the normal $i$-nodes. Then $\widetilde{f}_{i}(S_{s}(\lambda))$ is obtained by
changing in $S_{\boldsymbol{s}}(\boldsymbol{\lambda})$ the entry $x$ appearing
in $\widetilde{w}_{i}$ corresponding to the rightmost (normal) node into $x+1$
if $a>0$ and is zero otherwise. The symbol $S_{\boldsymbol{s}}%
(\boldsymbol{\emptyset})$ becomes a source vertex of highest weight
$\Lambda_{\mathbf{s}}$ and the associated connected component
$B(S_{\boldsymbol{s}}(\boldsymbol{\emptyset}))$ is isomorphic to
$B(\Lambda_{\mathbf{s}})$.\ Observe that both crystal structures
$B(S_{\boldsymbol{s}}(\boldsymbol{\emptyset}))$ and $B^{\infty}%
(S_{\boldsymbol{s}}(\boldsymbol{\emptyset}))$ are compatible: we have
$B(S_{\boldsymbol{s}}(\boldsymbol{\emptyset}))\subset B^{\infty}%
(S_{\boldsymbol{s}}(\boldsymbol{\emptyset)})$ (i.e. an inclusion of the sets
of vertices) and each arrow $S_{\boldsymbol{s}}(\boldsymbol{\lambda}%
)\overset{i}{\rightarrow}S_{\boldsymbol{s}}(\boldsymbol{\mu})$ in
$B(S_{\boldsymbol{s}}(\boldsymbol{\emptyset}))$ is an arrow $S_{\boldsymbol{s}%
}(\boldsymbol{\lambda})\overset{j}{\rightarrow}S_{\boldsymbol{s}%
}(\boldsymbol{\mu})$ in $B^{\infty}(S_{\boldsymbol{s}}(\boldsymbol{\emptyset
}))$ with $j=i\operatorname{mod}e$ where $W_{j}$ is the factor of $w_{i}$
modified in (\ref{w_idev}) when $\tilde{f}_{i}$ is applied to
$S_{\boldsymbol{s}}(\boldsymbol{\lambda})$.

\subsection{Orbit of the highest weight vertex}

\label{SubsecOrbHwv}Let $\mathbf{s}=(s_{1},\ldots,s_{l})\in\mathbb{Z}^{l}$ be
an arbitrary multicharge and $e\in\mathbb{Z}^{>0}$. We now give a
characterization of the $l$-partitions in the orbit $O(\mathbf{s},e)$ of
$S_{\boldsymbol{s}}(\boldsymbol{\emptyset})$ modulo the action of the affine
Weyl group similar to Corollary \ref{Cor_core}. To an $l$-partition
${\boldsymbol{\lambda}}$, we attach its abacus (which depends on $\mathbf{s}%
$). Recall that $L_{\boldsymbol{s}}(\boldsymbol{\lambda})=(L_{s_{l}}%
(\lambda^{l}),\ldots,L_{s_{1}}(\lambda^{1}))$ is the abacus of $\lambda$. For
two runners $L_{s}(\lambda)$ and $L_{t}(\mu)$ in one abacus, write
$L_{s}(\lambda)\subset L_{t}(\mu)$ when for each black bead in the runner
$L_{s}(\lambda)$, there is a black bead at the same position in the runner
$L_{t}(\mu)$. Alternatively, let $S_{s}(\lambda)$ and $S_{t}(\mu)$ be the
symbols corresponding to these two runners.\ We have $L_{s}(\lambda)\subset
L_{t}(\mu)$ if and only if $S_{s}(\lambda)$ is contained in $S_{t}(\mu)$. Note
that for all $k\in\mathbb{Z}$, we have $L_{s}(\lambda)\subset L_{t}(\mu)$ if
and only if $L_{s+k}(\lambda)\subset L_{t+k}(\mu)$. By a slight abuse of
notation, we shall say that $x$ is in the abacus $L_{\boldsymbol{s}%
}(\boldsymbol{\lambda})$ if and only if there is a black bead in position $x$
in one of its runner (equivalently, $x$ appears in a row of $S_{s}(\lambda)$).

For any partition $\mu$ and any $s\in\mathbb{Z}$, the abacus $L_{-s}(\mu^{t})$
of the transpose partition $\mu^{t}$ is obtained from $L_{s}(\mu)$ by
switching the black and white beads and performing a mirror image.\ Observe
also that $\lambda$ is an $e$-core if and only if $\lambda^{t}$ is. This is
easy to see on the abacus where it suffices to check that for all black bead
in position $x$ , there is a black bead in position $x-e$.

\begin{example}
\label{Exa_trans}Compare below the abaci of the partitions $(5,3,1)$ with
charge $2$ and $(3.2.2.1.1)$ with charge $-2$.%

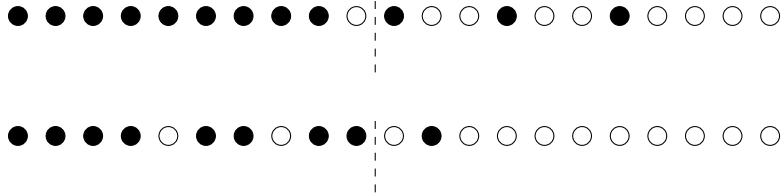
\[
\begin{tikzpicture}[scale=0.5, bb/.style={draw,circle,fill,minimum size=2.5mm,inner sep=0pt,outer sep=0pt}, wb/.style={draw,circle,fill=white,minimum size=2.5mm,inner sep=0pt,outer sep=0pt}]

	\node [wb] at (11,2) {};
	\node [wb] at (10,2) {};
	\node [wb] at (9,2) {};
	\node [wb] at (8,2) {};
	\node [bb] at (7,2) {};
	\node [wb] at (6,2) {};
	\node [wb] at (5,2) {};
	\node [bb] at (4,2) {};
	\node [wb] at (3,2) {};
	\node [wb] at (2,2) {};
	\node [bb] at (1,2) {};
	\node [wb] at (0,2) {};
	\node [bb] at (-1,2) {};
	\node [bb] at (-2,2) {};
	\node [bb] at (-3,2) {};
	\node [bb] at (-4,2) {};
	\node [bb] at (-5,2) {};
	\node [bb] at (-6,2) {};
	\node [bb] at (-7,2) {};
	\node [bb] at (-8,2) {};
	\node [bb] at (-9,2) {};

	\draw[dashed](0.5,0.5)--node[]{}(0.5,2.5);
	
	\end{tikzpicture}
\]

\[
\begin{tikzpicture}[scale=0.5, bb/.style={draw,circle,fill,minimum size=2.5mm,inner sep=0pt,outer sep=0pt}, wb/.style={draw,circle,fill=white,minimum size=2.5mm,inner sep=0pt,outer sep=0pt}]

	\node [wb] at (11,2) {};
	\node [wb] at (10,2) {};
	\node [wb] at (9,2) {};
	\node [wb] at (8,2) {};
	\node [wb] at (7,2) {};
	\node [wb] at (6,2) {};
	\node [wb] at (5,2) {};
	\node [wb] at (4,2) {};
	\node [wb] at (3,2) {};
	\node [bb] at (2,2) {};
	\node [wb] at (1,2) {};
	\node [bb] at (0,2) {};
	\node [bb] at (-1,2) {};
	\node [wb] at (-2,2) {};
	\node [bb] at (-3,2) {};
	\node [bb] at (-4,2) {};
	\node [wb] at (-5,2) {};
	\node [bb] at (-6,2) {};
	\node [bb] at (-7,2) {};
	\node [bb] at (-8,2) {};
	\node [bb] at (-9,2) {};

	\draw[dashed](0.5,0.5)--node[]{}(0.5,2.5);
	
	\end{tikzpicture}
\]

\end{example}

Assume $l>1$ and $\mathbf{s}=(s_{1},\ldots,s_{l})$. For any $1\leq a<b<l$, let
$s_{a}^{\prime}$ and $s_{b}^{\prime}$ be integers such that $s_{a}%
=s_{a}^{\prime}+p_{a}e,$ $s_{b}=s_{b}^{\prime}+p_{b}e$ with $(p_{a},p_{b}%
)\in\mathbb{Z}^{2}$ and $0\leq s_{a}^{\prime}-s_{b}^{\prime}<e$. When $l=1$,
we set $s_{1}^{\prime}=s_{1}$ for completeness. Note that $s_{a}^{\prime}$ and
$s_{b}^{\prime}$ are in fact defined modulo translation by the same multiple
of $e$ (one can see that that such a translation does not affect the
definition below ).

\begin{definition}
We say that the $l$-partition ${\boldsymbol{\lambda}}$ is a $(e,\mathbf{s}%
)$-core if it satisfies on of the following properties:

\begin{enumerate}
\item $l=1$ and $L_{s_{1}^{\prime}}(\lambda^{1})\subset L_{s_{1}^{\prime}%
+e}(\lambda^{1})$

\item $l>1$ and $L_{s_{a}^{\prime}}(\lambda^{a})\subset L_{s_{b}^{\prime}%
}(\lambda^{b})\subset L_{s_{a}^{\prime}+e}(\lambda^{a})$ for any $1\leq a<b<l$.
\end{enumerate}

We denote by $\mathfrak{L} (e,\mathbf{s})$ the set of all $(e,\mathbf{s})$-cores.
\end{definition}

\begin{remark}
\ 

\begin{enumerate}
\item The condition $L_{s_{a}^{\prime}}(\lambda^{a})\subset L_{s_{a}^{\prime
}+e}(\lambda^{b})$ means that for each $x$ in $L_{s_{a}^{\prime}}(\lambda
^{a}),$ $x-e$ also belongs to $L_{s_{a}^{\prime}}(\lambda^{a})$ (since $x\in
L_{s_{a}^{\prime}+e}(\lambda^{a})$). Thus, in the $(e,\mathbf{s})$-core
${\boldsymbol{\lambda}}$, each $\lambda^{a}$ is a core.

\item When $l>1$ and $\mathbf{s}=(s_{1},\ldots,s_{l})$ is such that $0\leq
s_{1}\leq\cdots\leq s_{l}<e$, then ${\boldsymbol{\lambda}}$ is a
$(e,\mathbf{s})$-core if and only if for any $1\leq a\leq l$, $\lambda^{a}$ is
a $e$-core and for any $1\leq a<l-1$, $L_{s_{a}}(\lambda^{a})\subset
L_{s_{a+1}}(\lambda^{b})\subset L_{s_{a}+e}(\lambda^{a})$.
\end{enumerate}
\end{remark}

\begin{example}
Assume $e=3$ and $\mathbf{s}=(0,1,1)$, then the $3$-partition $(1,3.1,3.1)$ is
in $\mathfrak{L}(e,\mathbf{s})$.

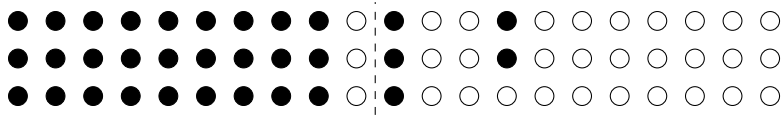
\[
\begin{tikzpicture}[scale=0.5, bb/.style={draw,circle,fill,minimum size=2.5mm,inner sep=0pt,outer sep=0pt}, wb/.style={draw,circle,fill=white,minimum size=2.5mm,inner sep=0pt,outer sep=0pt}]

			\node [wb] at (11,2) {};
			\node [wb] at (10,2) {};
			\node [wb] at (9,2) {};
			\node [wb] at (8,2) {};
			\node [wb] at (7,2) {};
			\node [wb] at (6,2) {};
			\node [wb] at (5,2) {};
			\node [bb] at (4,2) {};
			\node [wb] at (3,2) {};
			\node [wb] at (2,2) {};
			\node [bb] at (1,2) {};
			\node [wb] at (0,2) {};
			\node [bb] at (-1,2) {};
			\node [bb] at (-2,2) {};
			\node [bb] at (-3,2) {};
			\node [bb] at (-4,2) {};
			\node [bb] at (-5,2) {};
			\node [bb] at (-6,2) {};
			\node [bb] at (-7,2) {};
			\node [bb] at (-8,2) {};
			\node [bb] at (-9,2) {};
			
			\node [wb] at (11,1) {};
			\node [wb] at (10,1) {};
			\node [wb] at (9,1) {};
			\node [wb] at (8,1) {};
			\node [wb] at (7,1) {};
			\node [wb] at (6,1) {};
			\node [wb] at (5,1) {};
			\node [bb] at (4,1) {};
			\node [wb] at (3,1) {};
			\node [wb] at (2,1) {};
			\node [bb] at (1,1) {};
			\node [wb] at (0,1) {};
			\node [bb] at (-1,1) {};
			\node [bb] at (-2,1) {};
			\node [bb] at (-3,1) {};
			\node [bb] at (-4,1) {};
			\node [bb] at (-5,1) {};
			\node [bb] at (-6,1) {};
			\node [bb] at (-7,1) {};
			\node [bb] at (-8,1) {};
			\node [bb] at (-9,1) {};
			
			\node [wb] at (11,0) {};
			\node [wb] at (10,0) {};
			\node [wb] at (9,0) {};
			\node [wb] at (8,0) {};
			\node [wb] at (7,0) {};
			\node [wb] at (6,0) {};
			\node [wb] at (5,0) {};
			\node [wb] at (4,0) {};
			\node [wb] at (3,0) {};
			\node [wb] at (2,0) {};
			\node [bb] at (1,0) {};
			\node [wb] at (0,0) {};
			\node [bb] at (-1,0) {};
			\node [bb] at (-2,0) {};
			\node [bb] at (-3,0) {};
			\node [bb] at (-4,0) {};
			\node [bb] at (-5,0) {};
			\node [bb] at (-6,0) {};
			\node [bb] at (-7,0) {};
			\node [bb] at (-8,0) {};
			\node [bb] at (-9,0) {};

			\draw[dashed](0.5,-0.5)--node[]{}(0.5,2.5);
			
			\end{tikzpicture}
\]


\end{example}

\begin{example}
\label{exaff} Take $e=3$ and $\mathbf{s}=(0,1)$. The following are the
$2$-partitions of rank less than $3$ in $\mathfrak{L}(a,\mathbf{s})$:

The empty bipartition $(\emptyset,\emptyset)$, with abacus:%

\[
\begin{tikzpicture}[scale=0.5, bb/.style={draw,circle,fill,minimum size=2.5mm,inner sep=0pt,outer sep=0pt}, wb/.style={draw,circle,fill=white,minimum size=2.5mm,inner sep=0pt,outer sep=0pt}]

			\node [wb] at (11,2) {};
			\node [wb] at (10,2) {};
			\node [wb] at (9,2) {};
			\node [wb] at (8,2) {};
			\node [wb] at (7,2) {};
			\node [wb] at (6,2) {};
			\node [wb] at (5,2) {};
			\node [wb] at (4,2) {};
			\node [wb] at (3,2) {};
			\node [wb] at (2,2) {};
			\node [bb] at (1,2) {};
			\node [bb] at (0,2) {};
			\node [bb] at (-1,2) {};
			\node [bb] at (-2,2) {};
			\node [bb] at (-3,2) {};
			\node [bb] at (-4,2) {};
			\node [bb] at (-5,2) {};
			\node [bb] at (-6,2) {};
			\node [bb] at (-7,2) {};
			\node [bb] at (-8,2) {};
			\node [bb] at (-9,2) {};
			
			\node [wb] at (11,1) {};
			\node [wb] at (10,1) {};
			\node [wb] at (9,1) {};
			\node [wb] at (8,1) {};
			\node [wb] at (7,1) {};
			\node [wb] at (6,1) {};
			\node [wb] at (5,1) {};
			\node [wb] at (4,1) {};
			\node [wb] at (3,1) {};
			\node [wb] at (2,1) {};
			\node [wb] at (1,1) {};
			\node [bb] at (0,1) {};
			\node [bb] at (-1,1) {};
			\node [bb] at (-2,1) {};
			\node [bb] at (-3,1) {};
			\node [bb] at (-4,1) {};
			\node [bb] at (-5,1) {};
			\node [bb] at (-6,1) {};
			\node [bb] at (-7,1) {};
			\node [bb] at (-8,1) {};
			\node [bb] at (-9,1) {};
			
			\end{tikzpicture}
\]

The bipartition $(\emptyset,1)$, with abacus:%

\[
\begin{tikzpicture}[scale=0.5, bb/.style={draw,circle,fill,minimum size=2.5mm,inner sep=0pt,outer sep=0pt}, wb/.style={draw,circle,fill=white,minimum size=2.5mm,inner sep=0pt,outer sep=0pt}]

			\node [wb] at (11,2) {};
			\node [wb] at (10,2) {};
			\node [wb] at (9,2) {};
			\node [wb] at (8,2) {};
			\node [wb] at (7,2) {};
			\node [wb] at (6,2) {};
			\node [wb] at (5,2) {};
			\node [wb] at (4,2) {};
			\node [wb] at (3,2) {};
			\node [bb] at (2,2) {};
			\node [wb] at (1,2) {};
			\node [bb] at (0,2) {};
			\node [bb] at (-1,2) {};
			\node [bb] at (-2,2) {};
			\node [bb] at (-3,2) {};
			\node [bb] at (-4,2) {};
			\node [bb] at (-5,2) {};
			\node [bb] at (-6,2) {};
			\node [bb] at (-7,2) {};
			\node [bb] at (-8,2) {};
			\node [bb] at (-9,2) {};
			
			\node [wb] at (11,1) {};
			\node [wb] at (10,1) {};
			\node [wb] at (9,1) {};
			\node [wb] at (8,1) {};
			\node [wb] at (7,1) {};
			\node [wb] at (6,1) {};
			\node [wb] at (5,1) {};
			\node [wb] at (4,1) {};
			\node [wb] at (3,1) {};
			\node [wb] at (2,1) {};
			\node [wb] at (1,1) {};
			\node [bb] at (0,1) {};
			\node [bb] at (-1,1) {};
			\node [bb] at (-2,1) {};
			\node [bb] at (-3,1) {};
			\node [bb] at (-4,1) {};
			\node [bb] at (-5,1) {};
			\node [bb] at (-6,1) {};
			\node [bb] at (-7,1) {};
			\node [bb] at (-8,1) {};
			\node [bb] at (-9,1) {};
			
			\end{tikzpicture}
\]

The bipartition $(1,\emptyset)$, with abacus:%

\[
\begin{tikzpicture}[scale=0.5, bb/.style={draw,circle,fill,minimum size=2.5mm,inner sep=0pt,outer sep=0pt}, wb/.style={draw,circle,fill=white,minimum size=2.5mm,inner sep=0pt,outer sep=0pt}]

			\node [wb] at (11,2) {};
			\node [wb] at (10,2) {};
			\node [wb] at (9,2) {};
			\node [wb] at (8,2) {};
			\node [wb] at (7,2) {};
			\node [wb] at (6,2) {};
			\node [wb] at (5,2) {};
			\node [wb] at (4,2) {};
			\node [wb] at (3,2) {};
			\node [wb] at (2,2) {};
			\node [bb] at (1,2) {};
			\node [bb] at (0,2) {};
			\node [bb] at (-1,2) {};
			\node [bb] at (-2,2) {};
			\node [bb] at (-3,2) {};
			\node [bb] at (-4,2) {};
			\node [bb] at (-5,2) {};
			\node [bb] at (-6,2) {};
			\node [bb] at (-7,2) {};
			\node [bb] at (-8,2) {};
			\node [bb] at (-9,2) {};
			
			\node [wb] at (11,1) {};
			\node [wb] at (10,1) {};
			\node [wb] at (9,1) {};
			\node [wb] at (8,1) {};
			\node [wb] at (7,1) {};
			\node [wb] at (6,1) {};
			\node [wb] at (5,1) {};
			\node [wb] at (4,1) {};
			\node [wb] at (3,1) {};
			\node [wb] at (2,1) {};
			\node [bb] at (1,1) {};
			\node [wb] at (0,1) {};
			\node [bb] at (-1,1) {};
			\node [bb] at (-2,1) {};
			\node [bb] at (-3,1) {};
			\node [bb] at (-4,1) {};
			\node [bb] at (-5,1) {};
			\node [bb] at (-6,1) {};
			\node [bb] at (-7,1) {};
			\node [bb] at (-8,1) {};
			\node [bb] at (-9,1) {};
			
			\end{tikzpicture}
\]

The bipartition $(1.1,\emptyset)$, with abacus:%

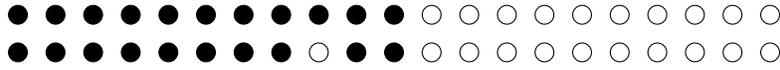
\[
\begin{tikzpicture}[scale=0.5, bb/.style={draw,circle,fill,minimum size=2.5mm,inner sep=0pt,outer sep=0pt}, wb/.style={draw,circle,fill=white,minimum size=2.5mm,inner sep=0pt,outer sep=0pt}]

			\node [wb] at (11,2) {};
			\node [wb] at (10,2) {};
			\node [wb] at (9,2) {};
			\node [wb] at (8,2) {};
			\node [wb] at (7,2) {};
			\node [wb] at (6,2) {};
			\node [wb] at (5,2) {};
			\node [wb] at (4,2) {};
			\node [wb] at (3,2) {};
			\node [wb] at (2,2) {};
			\node [bb] at (1,2) {};
			\node [bb] at (0,2) {};
			\node [bb] at (-1,2) {};
			\node [bb] at (-2,2) {};
			\node [bb] at (-3,2) {};
			\node [bb] at (-4,2) {};
			\node [bb] at (-5,2) {};
			\node [bb] at (-6,2) {};
			\node [bb] at (-7,2) {};
			\node [bb] at (-8,2) {};
			\node [bb] at (-9,2) {};
			
			\node [wb] at (11,1) {};
			\node [wb] at (10,1) {};
			\node [wb] at (9,1) {};
			\node [wb] at (8,1) {};
			\node [wb] at (7,1) {};
			\node [wb] at (6,1) {};
			\node [wb] at (5,1) {};
			\node [wb] at (4,1) {};
			\node [wb] at (3,1) {};
			\node [wb] at (2,1) {};
			\node [bb] at (1,1) {};
			\node [bb] at (0,1) {};
			\node [wb] at (-1,1) {};
			\node [bb] at (-2,1) {};
			\node [bb] at (-3,1) {};
			\node [bb] at (-4,1) {};
			\node [bb] at (-5,1) {};
			\node [bb] at (-6,1) {};
			\node [bb] at (-7,1) {};
			\node [bb] at (-8,1) {};
			\node [bb] at (-9,1) {};
			
			\end{tikzpicture}
\]

The bipartition $(\emptyset,2)$, with abacus:%

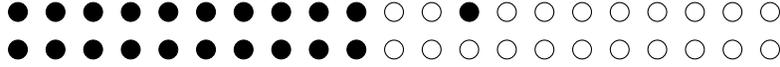
\[
\begin{tikzpicture}[scale=0.5, bb/.style={draw,circle,fill,minimum size=2.5mm,inner sep=0pt,outer sep=0pt}, wb/.style={draw,circle,fill=white,minimum size=2.5mm,inner sep=0pt,outer sep=0pt}]

			\node [wb] at (11,2) {};
			\node [wb] at (10,2) {};
			\node [wb] at (9,2) {};
			\node [wb] at (8,2) {};
			\node [wb] at (7,2) {};
			\node [wb] at (6,2) {};
			\node [wb] at (5,2) {};
			\node [wb] at (4,2) {};
			\node [bb] at (3,2) {};
			\node [wb] at (2,2) {};
			\node [wb] at (1,2) {};
			\node [bb] at (0,2) {};
			\node [bb] at (-1,2) {};
			\node [bb] at (-2,2) {};
			\node [bb] at (-3,2) {};
			\node [bb] at (-4,2) {};
			\node [bb] at (-5,2) {};
			\node [bb] at (-6,2) {};
			\node [bb] at (-7,2) {};
			\node [bb] at (-8,2) {};
			\node [bb] at (-9,2) {};
			
			\node [wb] at (11,1) {};
			\node [wb] at (10,1) {};
			\node [wb] at (9,1) {};
			\node [wb] at (8,1) {};
			\node [wb] at (7,1) {};
			\node [wb] at (6,1) {};
			\node [wb] at (5,1) {};
			\node [wb] at (4,1) {};
			\node [wb] at (3,1) {};
			\node [wb] at (2,1) {};
			\node [wb] at (1,1) {};
			\node [bb] at (0,1) {};
			\node [bb] at (-1,1) {};
			\node [bb] at (-2,1) {};
			\node [bb] at (-3,1) {};
			\node [bb] at (-4,1) {};
			\node [bb] at (-5,1) {};
			\node [bb] at (-6,1) {};
			\node [bb] at (-7,1) {};
			\node [bb] at (-8,1) {};
			\node [bb] at (-9,1) {};
			
			\end{tikzpicture}
\]

The bipartition $(1,1.1)$, with abacus:%

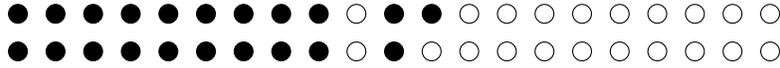
\[
\begin{tikzpicture}[scale=0.5, bb/.style={draw,circle,fill,minimum size=2.5mm,inner sep=0pt,outer sep=0pt}, wb/.style={draw,circle,fill=white,minimum size=2.5mm,inner sep=0pt,outer sep=0pt}]

			\node [wb] at (11,2) {};
			\node [wb] at (10,2) {};
			\node [wb] at (9,2) {};
			\node [wb] at (8,2) {};
			\node [wb] at (7,2) {};
			\node [wb] at (6,2) {};
			\node [wb] at (5,2) {};
			\node [wb] at (4,2) {};
			\node [wb] at (3,2) {};
			\node [bb] at (2,2) {};
			\node [bb] at (1,2) {};
			\node [wb] at (0,2) {};
			\node [bb] at (-1,2) {};
			\node [bb] at (-2,2) {};
			\node [bb] at (-3,2) {};
			\node [bb] at (-4,2) {};
			\node [bb] at (-5,2) {};
			\node [bb] at (-6,2) {};
			\node [bb] at (-7,2) {};
			\node [bb] at (-8,2) {};
			\node [bb] at (-9,2) {};
			
			\node [wb] at (11,1) {};
			\node [wb] at (10,1) {};
			\node [wb] at (9,1) {};
			\node [wb] at (8,1) {};
			\node [wb] at (7,1) {};
			\node [wb] at (6,1) {};
			\node [wb] at (5,1) {};
			\node [wb] at (4,1) {};
			\node [wb] at (3,1) {};
			\node [wb] at (2,1) {};
			\node [bb] at (1,1) {};
			\node [wb] at (0,1) {};
			\node [bb] at (-1,1) {};
			\node [bb] at (-2,1) {};
			\node [bb] at (-3,1) {};
			\node [bb] at (-4,1) {};
			\node [bb] at (-5,1) {};
			\node [bb] at (-6,1) {};
			\node [bb] at (-7,1) {};
			\node [bb] at (-8,1) {};
			\node [bb] at (-9,1) {};
			
			\end{tikzpicture}
\]

The bipartition $(2,1)$, with abacus:%

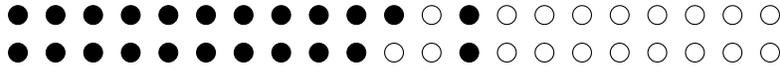
\[
\begin{tikzpicture}[scale=0.5, bb/.style={draw,circle,fill,minimum size=2.5mm,inner sep=0pt,outer sep=0pt}, wb/.style={draw,circle,fill=white,minimum size=2.5mm,inner sep=0pt,outer sep=0pt}]

			\node [wb] at (11,2) {};
			\node [wb] at (10,2) {};
			\node [wb] at (9,2) {};
			\node [wb] at (8,2) {};
			\node [wb] at (7,2) {};
			\node [wb] at (6,2) {};
			\node [wb] at (5,2) {};
			\node [wb] at (4,2) {};
			\node [bb] at (3,2) {};
			\node [wb] at (2,2) {};
			\node [bb] at (1,2) {};
			\node [bb] at (0,2) {};
			\node [bb] at (-1,2) {};
			\node [bb] at (-2,2) {};
			\node [bb] at (-3,2) {};
			\node [bb] at (-4,2) {};
			\node [bb] at (-5,2) {};
			\node [bb] at (-6,2) {};
			\node [bb] at (-7,2) {};
			\node [bb] at (-8,2) {};
			\node [bb] at (-9,2) {};
			
			\node [wb] at (11,1) {};
			\node [wb] at (10,1) {};
			\node [wb] at (9,1) {};
			\node [wb] at (8,1) {};
			\node [wb] at (7,1) {};
			\node [wb] at (6,1) {};
			\node [wb] at (5,1) {};
			\node [wb] at (4,1) {};
			\node [bb] at (3,1) {};
			\node [wb] at (2,1) {};
			\node [wb] at (1,1) {};
			\node [bb] at (0,1) {};
			\node [bb] at (-1,1) {};
			\node [bb] at (-2,1) {};
			\node [bb] at (-3,1) {};
			\node [bb] at (-4,1) {};
			\node [bb] at (-5,1) {};
			\node [bb] at (-6,1) {};
			\node [bb] at (-7,1) {};
			\node [bb] at (-8,1) {};
			\node [bb] at (-9,1) {};
			
			\end{tikzpicture}
\]

\end{example}

Given two multicharges $\mathbf{s}$ and $\mathbf{s}^{\prime}$ in
$\mathbb{Z}^{l}$, we have $\Lambda_{\mathbf{s}}=\Lambda_{\mathbf{s}^{\prime}}$
if and only if $\mathbf{s\operatorname{mod}e}$ and $\mathbf{s}^{\prime
}\operatorname{mod}e$ coincide up to permutation of their components. Then the
crystals $B(S_{\boldsymbol{s}}(\boldsymbol{\emptyset}))$ and
$B(S_{\boldsymbol{s}^{\prime}}(\boldsymbol{\emptyset}))$ are isomorphic.\ In
\cite{JL}, we establish that the associated isomorphism $\Phi_{\mathbf{s}%
\rightarrow\mathbf{s}^{\prime}}^{e}$ can always be obtained by composing two
types of elementary isomorphisms. The first one is denoted by $\Phi
_{\mathbf{s}\rightarrow(k,k+1)\cdot\mathbf{s}}^{e}$ and corresponds to the
permutation of $s_{k}$ and $s_{k+1}$ in $\mathbf{s}$. It is just the
restriction to the affine crystals of the combinatorial $\mathrm{R}$-matrix
$\Phi_{\mathbf{s}\rightarrow(k,k+1)\cdot\mathbf{s}}^{e}$ (which can be
computed by usual Jeu de Taquin operations). The second one is denoted
$\Phi_{\mathbf{s}\rightarrow\tau\cdot\mathbf{s}}^{e}$ where $\tau
\cdot\mathbf{s}=(s_{2},\ldots,s_{l},s_{1}+e)$ and sends the symbol
$S_{\mathbf{s}}(\boldsymbol{\lambda})$ on $S_{\tau\cdot\mathbf{s}%
}(\boldsymbol{\mu})$ where $\boldsymbol{\mu}=(\lambda^{2},\ldots,\lambda
^{l},\lambda^{1})$. Let us first prove our set $\mathfrak{L}(e,\mathbf{s})$ is
stable by these isomorphims.

\begin{lemma}
\label{iso} Let $\mathbf{s}\in\mathbb{Z}^{l}$ and $k\in\{1,\ldots,l-2\}$ be
such that $s_{k}\leq s_{k+1}$ then

\begin{enumerate}
\item We have $\Phi_{\mathbf{s}\rightarrow(k,k+1)\cdot\mathbf{s}}%
^{e}(\mathfrak{L}(e,\mathbf{s}))=\mathfrak{L}(e,\mathbf{s})$ and for any
$(e,\mathbf{s})$-core ${\boldsymbol{\lambda}}$, we get
\[
\Phi_{\mathbf{s}\rightarrow(k,k+1)\cdot\mathbf{s}}^{e}({\boldsymbol{\lambda}%
})=(\lambda^{1},\ldots,\lambda^{k+1},\lambda^{k},\ldots,\lambda^{l}).
\]

\item We have $\Phi_{\mathbf{s}\rightarrow\tau\cdot\mathbf{s}}^{e}%
(\mathfrak{L}(e,\mathbf{s}))=\mathfrak{L}(e,\mathbf{s})$ and for any
$(e,\mathbf{s})$-core ${\boldsymbol{\lambda}}$, we get
\[
\Phi_{\mathbf{s}\rightarrow\tau\cdot\mathbf{s}}^{e}({\boldsymbol{\lambda}%
})=(\lambda^{2},\ldots,\lambda^{l},\lambda^{1}).
\]

\end{enumerate}
\end{lemma}

\begin{proof}
Let $k\in\{1,\ldots,l-2\}$ be such that $s_{k}\leq s_{k+1}$. We first show
that if ${\boldsymbol{\lambda}}\in\mathfrak{L}(e,\mathbf{s})$, we have
$\Phi_{\mathbf{s}\rightarrow(k,k+1)\cdot\mathbf{s}}^{e}({\boldsymbol{\lambda}%
})=(\lambda^{1},\ldots,\lambda^{k+1},\lambda^{k},\ldots,\lambda^{l})$. Since
the computation of $\Phi_{\mathbf{s}\rightarrow\sigma_{k}\mathbf{s}}^{e}$
reduces to Jeu de Taquin, the equality $\Phi_{\mathbf{s}\rightarrow
(k,k+1)\cdot\mathbf{s}}^{e}({\boldsymbol{\lambda}})=(\lambda^{1}%
,\ldots,\lambda^{k+1},\lambda^{k},\ldots,\lambda^{l})$ is equivalent to the
condition $L_{s_{k}}(\lambda^{k})\subset L_{s_{k+1}}(\lambda^{k+1})$. Let
$(s_{k}^{\prime},s_{k+1}^{\prime})\in\mathbb{Z}^{2}$ and $(p_{k},p_{k+1}%
)\in\mathbb{Z}^{2}$ be such that $s_{k}=s_{k}^{\prime}+p_{k}e$ and
$s_{k+1}=s_{k+1}^{\prime}+p_{k+1}e$ and $0\leq s_{k+1}^{\prime}-s_{k}^{\prime
}<e$.

\noindent By definition we have $L_{s_{k}^{\prime}}(\lambda^{k})\subset
L_{s_{k+1}^{\prime}}(\lambda^{k+1})\subset L_{s_{k}^{\prime}+e}(\lambda^{k})$.
As $s_{k}\leq s_{k+1}$, we must have $p_{k}\leq p_{k+1}$. By hypothesis, we
have $L_{s_{k}^{\prime}+p_{k}e}(\lambda^{k})\subset L_{s_{k+1}^{\prime}%
+p_{k}e}(\lambda^{k+1})$. This gives $L_{s_{k}^{\prime}+p_{k}e}(\lambda
^{k})\subset L_{s_{k+1}^{\prime}+p_{k}e}(\lambda^{k+1})\subset L_{s_{k+1}%
^{\prime}+p_{k+1}e}(\lambda^{k+1})$ because $\lambda^{k+1}$ is an $e$-core
(see Assertion 5 of Proposition \ref{Prop_CharCore}). Thus $L_{s_{k}}%
(\lambda^{k})\subset L_{s_{k+1}}(\lambda^{k+1})$ as desired. We now show that
$\Phi_{\mathbf{s}\rightarrow(k,k+1)\cdot\mathbf{s}}^{e}({\boldsymbol{\lambda}%
})$ is a $(e,(k,k+1)\cdot\mathbf{s})$-core. So consider $(h_{k},h_{k+1}%
)\in\mathbb{Z}^{2}$ such that $s_{k+1}=s_{k}^{\prime\prime}+h_{k}^{\prime}e$
and $s_{k}=s_{k+1}^{\prime\prime}+h_{k+1}e$ where $0\leq s_{k+1}^{\prime
\prime}-s_{k}^{\prime\prime}<e$. Keeping the above notation, we have that
$s_{k}^{\prime\prime}=s_{k+1}^{\prime}$ and $s_{k+1}^{\prime\prime}%
=s_{k}^{\prime}+e$ (up to a translation by a the same integer). On the one
hand, we have $L_{s_{k+1}^{\prime}}(\lambda^{k+1})\subset L_{s_{k}^{\prime}%
+e}(\lambda^{k})$ and thus $L_{s_{k}^{\prime\prime}}(\lambda^{k+1})\subset
L_{s_{k+1}^{\prime\prime}}(\lambda^{k})$.\ On the second hand, we have
$L_{s_{k}^{\prime}}(\lambda^{k})\subset L_{s_{k+1}^{\prime}}(\lambda^{k+1})$
and thus $L_{s_{k+1}^{\prime\prime}-e}(\lambda^{k})\subset L_{s_{k}%
^{\prime\prime}}(\lambda^{k+1})$ or equivalently $L_{s_{k+1}^{\prime\prime}%
}(\lambda^{k})\subset L_{s_{k}^{\prime\prime}+e}(\lambda^{k+1})$. Finally, we
get the inclusion $\Phi_{\mathbf{s}\rightarrow(k,k+1)\cdot\mathbf{s}}%
^{e}(\mathfrak{L}(e,\mathbf{s}))\subset\mathfrak{L}(e,\mathbf{s})$ and
conclude that $\Phi_{\mathbf{s}\rightarrow(k,k+1)\cdot\mathbf{s}}%
^{e}(\mathfrak{L}(e,\mathbf{s}))=\mathfrak{L}(e,\mathbf{s})$ because
$\Phi_{\mathbf{s}\rightarrow(k,k+1)\cdot\mathbf{s}}^{e}$ is a crystal
isomorphism. This proves our first assertion.

\noindent For the second one, we use that $\Phi_{\mathbf{s}\rightarrow
\tau\cdot\mathbf{s}}^{e}({\boldsymbol{\lambda}})=(\lambda^{2},\ldots
,\lambda^{l},\lambda^{1})$. We just need to show that this is a $(e,\tau
.\mathbf{s})$-core. To do this, take $k\in\{2,\ldots,l\}$. Let $(s_{1}%
^{\prime},s_{k}^{\prime})\in\mathbb{Z}^{2}$ and $(p_{1},p_{k})\in
\mathbb{Z}^{2}$ be such that $s_{1}=s_{1}^{\prime}+p_{1}e$ and $s_{k}%
=s_{k}^{\prime}+p_{k}e$ and $0\leq s_{k}^{\prime}-s_{1}^{\prime}<e$. Then by
definition we have $L_{s_{1}^{\prime}}(\lambda^{1})\subset L_{s_{k}^{\prime}%
}(\lambda^{k})\subset L_{s_{1}^{\prime}+e}(\lambda^{1})$. Also $s_{1}%
+e=s_{1}^{\prime}+e+p_{1}e$ and $s_{k}=s_{k}^{\prime}+p_{k}e$ with
$0\leq(s_{1}^{\prime}+e)-s_{k}^{\prime}<e$. Thus it suffices to see that
$L_{s_{k}^{\prime}}(\lambda^{k})\subset L_{s_{1}^{\prime}+e}(\lambda
^{1})\subset L_{s_{k}^{\prime}+e}(\lambda^{k})$ which is clear from the above
property. Again, we obtain that $\Phi_{\mathbf{s}\rightarrow\tau
\cdot\mathbf{s}}^{e}(\mathfrak{L}(e,\mathbf{s}))=\mathfrak{L}(e,\mathbf{s})$.
\end{proof}

\begin{lemma}
\label{ue}For any $\mathbf{s}\in\mathbb{Z}^{l}$ the empty $l$-partition is a
$(e,\mathbf{s})$-core.
\end{lemma}

\begin{proof}
Assume that for $\mathbf{s}\in\mathbb{Z}^{l}$ and $a=1,\ldots,l-1$, $b>a$ we
have $0\leq s_{b}-s_{a}<e$, then it is clear that $L_{s_{a}}(\emptyset)\subset
L_{s_{b}}(\emptyset)\subset L_{s_{a}+e}(\emptyset)$.

Now let us consider any $\mathbf{s}\in\mathbb{Z}^{l}$. Let $a=1,\ldots,l-1$
and let $b>a$. Write as in the definition, $s_{a}=s_{a}^{\prime}+p_{a}e$ and
$s_{b}=s_{b}^{\prime}+p_{b}e$ with $(p_{a},p_{b})\in\mathbb{Z}^{2}$ such that
$0\leq s_{b}^{\prime}-s_{a}^{\prime}<e$. Then we have that $L_{s_{a}^{\prime}%
}(\emptyset)\subset L_{s_{b}^{\prime}}(\emptyset)\subset L_{s_{b}^{\prime}%
+e}(\emptyset)$ by the previous case. The result follows.
\end{proof}

\begin{lemma}
\label{ue2} Let $\mathbf{s}\in\mathbb{Z}^{l}$ and assume that
${\boldsymbol{\lambda}}$ is a $(e,\mathbf{s})$-core. Then
${\boldsymbol{\lambda}}^{t}:=((\lambda^{l})^{t},\ldots,(\lambda^{1})^{t})$ is
a $(e,(-s_{l},\ldots,-s_{1}))$-core.
\end{lemma}

\begin{proof}
Assume that for any $a=1,\ldots,l-1$ and $b>a$ we have $0\leq s_{b}-s_{a}%
<e$.\ Then for the multicharge $\mathbf{t}:=(-s_{l},\ldots,-s_{1})$ we also
have for any $a=1,\ldots,l-1$ and $b>a$ the inequalities $0\leq t_{b}-t_{a}%
<e$. Our result then follows from the interpretation of the transposition on
abaci (see Example \ref{Exa_trans}).
\end{proof}

We can now describe the orbit $O(\mathbf{s},e)$ in the Uglov realization of
the crystal $B(\Lambda_{\mathbf{s}})$.

\begin{proposition}
\label{Th_Orbit_affine}Let $\mathbf{s}\in\mathbb{Z}^{l}.$ The $l$-partition
${\boldsymbol{\lambda}}$ yields a symbol in $O(\mathbf{s},e)$ if and only it
is a $(e,\mathbf{s})$-core. In particular, when $l>1$ and $0\leq s_{1}%
\leq\cdots\leq s_{l}<e$, the symbols in $O(\mathbf{s},e)$ are exactly those
such that $L_{s_{a}}(\lambda^{a})\subset L_{s_{a+1}}(\lambda^{b})\subset
L_{s_{a}+e}(\lambda^{a})$ for any $1\leq a<l$.
\end{proposition}

We shall first need the following lemma.


\begin{lemma}
\label{rem} Let $\mathbf{s}\in\mathbb{Z}^{l}$ and ${\boldsymbol{\lambda}}$ be
a $(e,\mathbf{s})$-core. Assume that we have a removable $j$-node at the
position $x$ of the abacus $S_{\boldsymbol{s}}(\boldsymbol{\lambda})$ (thus
$x\equiv j\operatorname{mod}e$)$.$ Then there is no addable $j$-node in
${\boldsymbol{\lambda}}$.
\end{lemma}

\begin{proof}
Assume first $l=1$. Since we have a removable node at position $x$, $x$ lies
in the abacus $S_{\boldsymbol{s}}(\boldsymbol{\lambda})$ but not $x-1$. As a
consequence, for all $a\in\mathbb{Z}_{>0}$, $x-1+ae$ does not belong to
$S_{\boldsymbol{s}}(\boldsymbol{\lambda})$ which thus has no addable node
greater than $x$. In addition, for all $a\in\mathbb{Z}_{>0}$, the node $x-a.e$
is in the abacus and this implies that there is no addable node in the abacus.

\noindent Now, assume $l>1$ and $x$ is a removable $j$-node on $k$-th runner
of ${\boldsymbol{\lambda}}$. When $k\neq1$, we have $\Phi_{\mathbf{s}%
\rightarrow\tau\cdot\mathbf{s}}^{e}({\boldsymbol{\lambda}})=(\lambda
^{2},\ldots,\lambda^{l},\lambda^{1})$ which is a $(e,\tau\cdot\mathbf{s}%
)$-core by Lemma \ref{iso}. Therefore, $x$ is a removable $j$-node in the
$k-1$-runner of $\Phi_{\mathbf{s}\rightarrow\sigma_{i}\mathbf{s}}%
^{e}({\boldsymbol{\lambda}})$. Clearly, $\Phi_{\mathbf{s}\rightarrow\tau
\cdot\mathbf{s}}^{e}({\boldsymbol{\lambda}})$ has no addable $j$-node if and
only if this holds in ${\boldsymbol{\lambda}}$. By repeating this argument we
can restrict the proof to the case $k=1$.

\noindent First, by the same arguments as in the case $l=1$, there is no
addable $j$-node in $L_{s_{1}^{\prime}}(\mu^{1})$. The condition
$L_{s_{1}^{\prime}}(\mu^{1})\subset L_{s_{b}^{\prime}}(\lambda^{b})$ for any
$b=2,\ldots,l$ implies that the $j$-node $x$ belongs to the $b$-th runner of
$L_{\boldsymbol{s}}(\boldsymbol{\lambda})$. If it is removable, we get as in
the case $l=1$ that there is no addable node in $L_{s_{b}^{\prime}}%
(\lambda^{b})$ since $\lambda^{b}$ is an $e$-core. If not, there is no addable
node in $L_{s_{b}^{\prime}}(\lambda^{b})$ which is less than $x$ because all
the positions $x-a.e$ and $x-1-a.e$ with $a>0$ are occupied. Also the
condition $L_{s_{b}}(\lambda^{b})\subset L_{s_{1}+e}(\lambda^{1})$ implies
that there is no node $x-1+e$ in $L_{s_{b}^{\prime}}(\lambda^{b})$.\ Otherwise
$x-1$ would belong to $L_{s_{1}^{\prime}}(\lambda^{1})$ and $x$ could not be
removable in $L_{s_{1}^{\prime}}(\lambda^{1})$. Thus, there is also no addable
$j$-node greater than $x$ in $L_{s_{b}^{\prime}}(\lambda^{b})$ for any
$b=2,\ldots,l.$ Finally we have showed there is no addable $j$-node in the
runners $L_{s_{a}^{\prime}}(\lambda^{a}),a=1,\ldots,l$. Since $s_{a}%
=s_{a}^{\prime}\operatorname{mod}e$ for any $a=1,\ldots,l$, this is also true
for the runners $L_{s_{a}}(\lambda^{a}),a=1,\ldots,l$.
\end{proof}

\bigskip

\begin{proof}
[Proof of Proposition \ref{Th_Orbit_affine}]Let us prove first the inclusion
$\mathfrak{L}(e,\mathbf{s})\subset O(\mathbf{s},e)$. Consider
${\boldsymbol{\lambda}}\in\mathfrak{L}(e,\mathbf{s})$, we show that
${\boldsymbol{\lambda}}$ is in $O(\mathbf{s},e)$ by induction on the rank $n$
of ${\boldsymbol{\lambda}}$.\ For $n=0$, the result is true by Lemma \ref{ue}.

Assume that $n>0$.\ Then ${\boldsymbol{\lambda}}$ is non empty and there
exists a removable node for ${\boldsymbol{\lambda}}$. Let $j$ be its residue.
As ${\boldsymbol{\lambda}}$ has no addable $j$-node by Lemma \ref{rem}, all
the removable nodes with residue $j$ are normal nodes. If we remove them, the
resulting $l$-partition ${\boldsymbol{\mu}}$ is clearly in $\mathfrak{L}%
(e,\mathbf{s})$ and thus also in $O(\mathbf{s},e)$ by induction.\ Then, adding
to ${\boldsymbol{\mu}}$ all the normal addable $j$-nodes gives
${\boldsymbol{\lambda}}$ which is thus in $O(\mathbf{s},e)$.


To prove the inclusion $\mathfrak{L}(e,\mathbf{s})\supset O(\mathbf{s},e)$,
consider $\boldsymbol{\lambda}$ such that $S_{\boldsymbol{s}}%
(\boldsymbol{\lambda})\in O(\mathbf{s},e)$. By (\ref{i-chain_orbit}), there
exists at least an integer $j\in\{1,\ldots,e-1\}$ such that
$\boldsymbol{\lambda}$ contains only removable $j$-nodes.\ Then
$S_{\boldsymbol{s}}(\boldsymbol{\mu})=s_{j}\cdot S_{\boldsymbol{s}%
}(\boldsymbol{\lambda})\in O(\mathbf{s},e)$ and by the induction hypothesis,
$\boldsymbol{\mu\in}\mathfrak{L}(e,\mathbf{s})$. Then adding to
$\boldsymbol{\mu}$ all its addable $j$-nodes gives the $l$-partition
$\boldsymbol{\lambda}$ in $\mathfrak{L}(e,\mathbf{s})$.
\end{proof}

\begin{example}
Let us resume Example \ref{exaff}. Denote by $\{s_{0},s_{1},s_{2}\}$ the
simple reflections of the affine Weyl group $\widehat{W}_{3}$. We have:
\[
S_{(0,1)}(\emptyset,1)=s_{1}S_{(0,1)}(\emptyset,\emptyset),\ S_{(0,1)}%
(1,\emptyset)=s_{0}S_{(0,1)}(\emptyset,\emptyset),\ S_{(0,1)}(1.1,\emptyset
)=s_{2}.s_{0}S_{(0,1)}(\emptyset,\emptyset)
\]%
\[
S_{(0,1)}(\emptyset,2)=s_{2}s_{1}S_{(0,1)}(\emptyset,\emptyset),\ S_{(0,1)}%
(1,1.1)=s_{0}s_{1}S_{(0,1)}(\emptyset,\emptyset),\ S_{(0,1)}(2,1)=s_{1}%
s_{0}S_{(0,1)}(\emptyset,\emptyset).
\]

\end{example}

\subsection{More on $(e,\mathbf{s})$-cores}

We here point an interesting property of the set of $(e,\mathbf{s})$-cores.
Assume that $\mathbf{s}$ is an arbitrary multicharge. Following \cite{F}, for
any $l$-partition ${\boldsymbol{\lambda}}$ and for each pairs of integers
$(i,j)\in\{0,\ldots,e-1\}\times\{1,\ldots,l\}$ set%
\[
b_{i,j}^{\mathbf{s}}({\boldsymbol{\lambda}}):=\operatorname{max}(\beta\in
S_{s_{j}}(\lambda^{j})|\beta\equiv i\operatorname{mod}e).
\]
Now let ${\widetilde{\mathbf{s}}}=(\widetilde{s}_{1},\ldots,\widetilde{s}%
_{l})\in\{0,1,\ldots,e-1\}^{l}$ be such that $\mathbf{s}\equiv{\widetilde
{\mathbf{s}}\operatorname{mod}e}$.

\begin{proposition}
For any ${\boldsymbol{\lambda}}\in O(\mathbf{s},e)$ and any $i\in
\{0,1,\ldots,e-1\}$ there exists an integer $\gamma_{i}$ such that
$b_{i,j}^{\widetilde{\mathbf{s}}}({\boldsymbol{\lambda}})\in\{\gamma
_{i},\gamma_{i}+e\}$ for all $j\in\{1,\ldots,l\}$.
\end{proposition}

\begin{proof}
Fix $i\in\{0,1,\ldots,e-1\}$. To prove the proposition, it suffices to show
that for all $(j,k)\in\{1,\ldots,l\}^{2}$, we have that $|b_{i,j}
^{\widetilde{\mathbf{s}}}({\boldsymbol{\lambda}})-b_{i,k}^{\widetilde
{\mathbf{s}}}({\boldsymbol{\lambda}})|\in\{0,e\}$. Assume first that
$\widetilde{s}_{j}\leq\widetilde{s}_{k}$. In this case, since
${\boldsymbol{\lambda}}\in O(\mathbf{s},e)$, we have $L_{\widetilde{s}_{j}%
}(\lambda^{j})\subset L_{\widetilde{s}_{k}}(\lambda^{k})\subset L_{\widetilde
{s}_{j}+e}(\lambda^{j})$ and this implies that $b_{i,k}^{\widetilde
{\mathbf{s}} }({\boldsymbol{\lambda}})-b_{i,j}^{\widetilde{\mathbf{s}}}
({\boldsymbol{\lambda}})\in\{0,e\}$. If we have now $\widetilde{s}_{j}
\geq\widetilde{s}_{k}$, we thus have $\widetilde{s}_{j}\leq\widetilde{s}
_{k}+e$ and we get $L_{\widetilde{s}_{j}}(\lambda^{j})\subset L_{\widetilde
{s}_{k}+e}(\lambda^{k})\subset L_{\widetilde{s}_{j}+e}(\lambda^{j})$. This
implies now that $b_{i,j}^{\widetilde{\mathbf{s}}}({\boldsymbol{\lambda}%
})-b_{i,k}^{\widetilde{\mathbf{s}}}({\boldsymbol{\lambda}})\in\{0,e\}$ as desired.
\end{proof}

\bigskip

It follows from the previous proposition and results in \cite[Th. 3.1]{F} that
the $(e,\mathbf{s})$-cores parametrize distinguished elements of certain
remarkable blocks for Ariki-Koike algebras which may be seen as analogues of
simple blocks for Iwahori-Hecke algebras. These elements give in fact  analogues of the $e$-cores 
 in the context of Ariki-Koike algebras, this will be developed in  \cite{JLw}.

\subsection{Strong Bruhat order on Keys}

Consider $S_{\boldsymbol{s}}(\boldsymbol{\lambda})\in O(\mathbf{s},e)$ and
$S_{\boldsymbol{s}}(\boldsymbol{\mu})\in O(\mathbf{s},e)$. Let $u$ and $v$ be
the elements in $W^{\Lambda_{\boldsymbol{s}}}$ such that $S_{\boldsymbol{s}%
}(\boldsymbol{\lambda})=u\cdot S_{\boldsymbol{s}}(\boldsymbol{\emptyset})$ and
$S_{\boldsymbol{s}}(\boldsymbol{\mu})=v\cdot S_{\boldsymbol{s}}%
(\boldsymbol{\emptyset})$, respectively. Recall we have written $\subseteq$
for the inclusion order on partitions. Since the higher level Uglov
$\widehat{\mathfrak{sl}}_{e}$-crystal structure on the set of symbols is not a
tensor product of level $1$ $\widehat{\mathfrak{sl}}_{e}$-crystals, we cannot
directly use the results of Section \ref{Sec_Th_fund}.\ Nevertheless, we can
get the following description of the strong Bruhat order on $O(\mathbf{s},e)$.

\begin{proposition}
\label{Prop_SBO_lpartitions}With the previous notation, we have
$u\trianglelefteq v$ if and only of $\lambda^{k}\subseteq\mu^{k}$ for any
$k=1,\ldots,l$.
\end{proposition}

\begin{proof}
Recall that each symbol $S_{\boldsymbol{s}}(\boldsymbol{\nu})$ can be regarded
as a vertex of the type $A_{e-1}^{(1)}$ and $A_{\infty}$ crystals
$B(S_{\boldsymbol{s}}(\boldsymbol{\emptyset}))$ and $B^{\infty}%
(S_{\boldsymbol{s}}(\boldsymbol{\emptyset}))$. Now consider the finite set
$S_{\boldsymbol{s}}^{N}(\boldsymbol{\nu})$ of symbols of $l$-partitions
$\boldsymbol{\nu}$ with rank less or equal to a fixed integer $N$,. Then,
there exists an integer $m$ such the action of any simple reflection $s_{i}%
\in\widehat{W}_{e}$ on $S_{\boldsymbol{s}}^{N}(\boldsymbol{\nu})$ coincide
with that of the permutation $\sigma_{i}:=\prod_{-m\leq k\leq m}%
(i+ke,i+1+ke)\in S_{[-m,m]}\subset W_{\infty}$ where $S_{[-m,m]}$ is the
symmetric group on the integers between $-m$ and $m$.\ More generally, the
action of $w\in\widehat{W}_{e}$ with minimal decomposition $w=s_{i_{1}}\cdots
s_{i_{a}},$ on $S_{\boldsymbol{s}}^{N}(\boldsymbol{\nu})$ will coincide with
that of $\widetilde{w}=\sigma_{i_{1}}\cdots\sigma_{i_{a}}\in S_{[-m,m]}$ and
$\sigma_{i_{1}}\cdots\sigma_{i_{a}}$ is also a minimal decomposition of
$\widetilde{w}$.\ By definition of the strong Bruhat order we have
$u\trianglelefteq v$ in $\widehat{W}_{e}$ if and only if $\widetilde
{u}\trianglelefteq\widetilde{v}$ in $S_{[-m,m]}$.\ Therefore, we are reduced
to the finite type $A$ for which the strong Bruhat order of level $l$ is the
product of $l$ strong Bruhat orders of level $1$ as observed in \S \ref{ftyA}.
\end{proof}

\subsection{Kleshchev realization and computation of the Keys}

\label{SubsecKles}As we have explained in Section \ref{Sec_Th_fund}, the
computation of the Keys for an element in $B(\Lambda_{\mathbf{s}})$ can be
reduced to the computation of the Keys for the crystals $B(\omega_{s_{j}})$
associated to the fundamental highest weights once the orbit of the highest
weight vertex in $B(\Lambda_{\mathbf{s}})$ and the combinatorial $R$-matrices
between fundamental crystals can be described. However, to do this the crystal
$B(\Lambda_{\mathbf{s}})$ should be realized as a connected component in a
tensor product of level $1$ crystals. Unfortunately, this is not the case for
the Uglov realization when $e$ is finite. The realization relevant in order to
use Corollary \ref{Cor_fund} is the Kleshchev one (see for example
\cite[\S 6.2.16]{GJ}). The vertices of the associated crystal are called the
Kleshchev multipartitions. They have, in principle, a non trivial inductive
definition but an elementary characterization of them has been recently given
in \cite{Jn}.

Fortunately, Uglov and Kleshchev realizations can be easily connected. In
particular, one can deduce from the above results that the characterization of
the multipartitions in the orbit of the empty multipartition are the same in
Kleshchev and Uglov realizations.

\begin{proposition}
\label{Th_Orbit_affineK}Let $\mathbf{s}\in\mathbb{Z}^{l}.$ In the Kleshchev
realization of $B(\Lambda_{\mathbf{s}})$, a $l$-partition
${\boldsymbol{\lambda}}$ yields a symbol in the orbit of the empty
$l$-partition if and only it is a $(e,\mathbf{s})$-core. In particular, when
$l>1$ and $0\leq s_{1}\leq\cdots\leq s_{l}<e$, the symbols in this orbit are
exactly those such that $L_{s_{a}}(\lambda^{a})\subset L_{s_{a+1}}(\lambda
^{b})\subset L_{s_{a}+e}(\lambda^{a})$ for any $1\leq a<l$%
.\footnote{Therefore, ${\boldsymbol{\lambda}}$ is a $(e,\mathbf{s})$-core if
and only if for any $1\leq a\leq l$, $\lambda^{a}$ is a $e$-core and for any
$1\leq a<l-1$, $L_{s_{a}}(\lambda^{a})\subset L_{s_{a+1}}(\lambda^{b})\subset
L_{s_{a}+e}(\lambda^{a})$.}
\end{proposition}

\begin{proof}
Let $n\in\mathbb{Z}_{>0}$. Take $\mathbf{t}=(t_{1},\ldots,t_{l})\in
\mathbb{Z}^{l}$ such that $t_{j}\equiv s_{j}(\text{mod }e)$ for all
$j=1,\ldots,l$ and such that $t_{j}-t_{1}\geq n+e$ for all $j=2,\ldots,l$. By
\cite[\S 6.2.16]{GJ}, the subcrystals containing the multipartitions of rank
less than $n$ in the Kleshchev realization for the multicharge $\mathbf{s}$
and in the Uglov realization for the multicharge $\mathbf{t}$ coincide. Thus
we can conclude \ by using the fact that $\mathbf{t}^{\prime}=\mathbf{s}%
^{\prime}$.
\end{proof}


Now the second crucial ingredient in our procedure for computing the key by
reduction to the fundamental weights as prescribed by \ref{Cor_fund} is the
combinatorial $R$-matrix associated to a pair of fundamental weights. It
corresponds to a transposition $(k,k+1)$ in the multicharge $\mathbf{s}$ for
the Kleshchev realization of crystals (the rank of the multipartitions being
fixed). Since only the components $k$ and $k+1$ are affected by this
$R$-matrix, we are reduced to the the case where $\mathbf{s}=(s_{1},s_{2})$
and $k=1$.

\noindent Let $\mathbf{v}=(v_{1},v_{2})\in\mathbb{Z}^{2}$ be such that $0\leq
v_{1}\leq v_{2}<e$ and $v_{j}\equiv s_{j}(\text{mod }e)$ for $j=1,2$. Then the
subcrystal containing the multipartitions of rank less than $n$ in the
Kleshchev realization for $\mathbf{s}=(s_{1},s_{2})$ coincides with that in
the Uglov realization for the multicharge $\mathbf{v}^{>}:=(v_{1},v_{2}+ke)$
where $k\in\mathbb{N}$ is such that $k.e>n+e$. The desired $R$-matrix thus
corresponds to a crystal isomorphism between the crystal associated to the
multicharge $\mathbf{v}^{>}$ and the crystal associated to the multicharge
$(v_{2},v_{1}+ke)$. This isomorphism can be computed on the bipartition
$(\lambda^{1},\lambda^{2})$ by composing the crystal isomorphisms described in
\S \ref{SubsecOrbHwv} as follows.

\begin{enumerate}
\item First apply the crystal isomorphism $\Phi_{(v_{1},v_{2}+ke)\rightarrow
(v_{2}+ke,v_{1}+e)}^{e}$ which exchanges the two components of the
bipartition, that is exchanges the two rows in the symbols and next translates
the bottom one by $e$.

\item Apply the crystal isomorphism $\Phi_{(v_{2}+ke,v_{1}+e)\rightarrow
(v_{1}+e,v_{2}+k.e)}^{e}$ which reduces to a \textquotedblleft Jeu de
taquin\textquotedblright\ switching the lengths of the two rows in the symbols.

\item Repeat the two previous steps $2k$ times to get the image of
$(\lambda^{1},\lambda^{2})$ in the crystal with multicharge $(v_{1}%
+2e,v_{2}+k.e)$.\ 

\item Finally, apply one more isomorphism $\Phi_{(v_{1}+2ke,v_{2}%
+ke)\rightarrow(v_{2}+ke,v_{1}+2ke+e)}^{e}$ and use the fact that the
isomorphism between the crystals with multicharge $(v_{2}+ke,v_{1}+2ke+e)$ in
the Uglov realization and $(s_{2},s_{1})$ in the Kleshchev realization is trivial.
\end{enumerate}

\begin{remark}
\label{Rem_keyUglov} \ The crystal isomorphism between the Uglov and Kleshchev
realizations of $B(\Lambda_{\mathbf{s}})$ can also be obtained from the
results in \cite{JL} although it is not easy to make explicit. By 2 of Remark
\ref{Rq_impo}, we so obtain a characterization of the Demazure crystals in the
Uglov realization. Nevertheless, we can just use the Kleshchev realization in
which is the orbit of the highest weight and the relevant combinatorial
$R$-matrices are easy to describe.
\end{remark}

\subsection{Generalization of the Young Lattice}

When $e=\infty,$ $l=1$ and $\mathbf{s}=(0),$ the orbit $O(\mathbf{s},e)$
coincides with the whole crystal $B(S_{\boldsymbol{s}}(\boldsymbol{\emptyset
}))$. By forgetting the labels $i$ of the arrows in $B(S_{\boldsymbol{s}%
}(\boldsymbol{\emptyset}))$, one then recovers the Young lattice $\mathcal{Y}$
of partitions which is strongly connected with the combinatorics of Schur
functions. This lattice admits an interesting generalization $\mathcal{Y}%
_{e-1}$ where the ordinary partitions are replaced by the $e$-cores (or by the
$k$ bounded partitions with $k=e-1$) connected this times with the
combinatorics of $k$-Schur functions (see \cite{LLMSSZ}).\ The graph
$\mathcal{Y}_{e-1}$ corresponds to the Hasse diagram of the orbit
$O(\mathbf{s},e)$ when $l=1$ and $\mathbf{s}=(0)$ and we have an arrow
$\lambda\rightarrow\mu$ between the two $e$-cores $\lambda$ and $\mu$ when
$\mu$ is obtained by adding all the possible addable $i$-nodes in $\lambda$
corresponding to a fixed $i\in\{0,\ldots,e-1\}$. When $l>1$, the notion of
$(e,\mathbf{s})$-core yields generalizations of the graph $\mathcal{Y}_{e-1}$
whose structure is obtained similarly from the orbit $O(\mathbf{s},e)$.\ It is
a natural question to ask whether its combinatorial properties (for $e$ finite
or not) can also be encoded in the combinatorics of a distinguished basis in a
polynomial algebra analogous to $k$-Schur functions in level $1$.

\section{Demazure crystals in $B(\infty)$}

\subsection{Link with the Demazure crystals in $B(\lambda)$}

Consider $\mathfrak{g}$ a Kac-Moody algebra and $\lambda$ a dominant weight
for $\mathfrak{g}$.\ We now explain how it is possible to characterize the
elements of a Demazure crystal $B(\infty)_{w}$ from the computation of keys in
the crystals $B(\lambda)_{w},\lambda\in P_{+}$.\ First recall that there
exists a unique embedding%
\[
\pi_{\lambda}:\left\{
\begin{array}
[c]{c}%
B(\lambda)\hookrightarrow B(\infty)\\
b\longmapsto\pi_{\lambda}(b)
\end{array}
\right.
\]
such that for any path $b=\tilde{f}_{i_{1}}\cdots\tilde{f}_{i_{l}}b_{\lambda}$
we have $\pi_{\lambda}(b)=\tilde{f}_{i_{1}}\cdots\tilde{f}_{i_{l}}b_{0}$ where
$b_{0}$ is the highest weight vertex of $B(\infty)$.\ Also the crystal
$B(\infty)$ is endowed with the Kashiwara involution $\ast$ and the crystal
operators have starred versions $\tilde{f}_{i}^{\ast}=\ast\circ\tilde{f}%
_{i}\circ\ast$ and $\tilde{e}_{i}^{\ast}=\ast\circ\tilde{e}_{i}\circ\ast
$.\ Thanks to the operators $\tilde{e}_{i}^{\ast}$, we get a simple
characterization of the image of $\pi_{\lambda}$. Namely, we have%
\[
\operatorname{Im}\pi_{\lambda}=\{u\in B(\infty)\mid\varepsilon^{\ast
}(u)\preceq\lambda\}
\]
where $\varepsilon^{\ast}(u)=\sum_{i\in\in I}\varepsilon_{i}^{\ast}%
(u)\omega_{i}$ and $\varepsilon^{\ast}(u)\preceq\lambda$ means that
$\lambda-\varepsilon^{\ast}(u)$ is a dominant weight.

Given any $w$ in the Weyl group $W$, we also have by Theorem \ref{Th_KL}%
\[
\pi_{\lambda}(B(\lambda)_{w})=B(\infty)_{w}\cap\operatorname{Im}\pi_{\lambda}.
\]
From the previous considerations, for deciding if a vertex $u$ belongs to
$B(\infty)_{w}$, it suffices to have a realization of $B(\lambda)$ and
$B(\infty)$ satisfying the properties below.

\begin{itemize}
\item The embedding $\pi_{\lambda}$ is easy to describe.

\item The actions of both the ordinary and $\ast$-crystal operators are explicit.

\item For any $u\in\operatorname{Im}\pi_{\lambda}$, one can compute the unique
vertex $b\in B(\lambda)$ such that $\pi_{\lambda}(b)=u$.

\item One can decide if a vertex $b$ in $B(\lambda)$ belongs to $B(\lambda
)_{w}$.
\end{itemize}

For deciding wether $u\in B(\infty)_{w}$ it then suffices to proceed as follows.

\begin{enumerate}
\item Compute $\lambda=\varepsilon^{\ast}(u)$, we get that $u\in
\operatorname{Im}\pi_{\lambda}$.

\item Determine $b\in B(\lambda)$ such that $\pi_{\lambda}(b)=u$.

\item Then, $u\in B(\infty)_{w}$ if and only if $b\in B(\lambda)_{w}$.
\end{enumerate}

\subsection{Finite, infinite and affine type $A$}

In type $A$ (finite, infinite and affine), vertices of $B(\infty)$ are
parametrized by multisegments that we now define.

\begin{definition}
A {segment} is a sequence of consecutive integers $[a,a+1,...,b]$. We denote
it by $[a;b]$. A collection (or a formal sum) of segments is called a
{multisegment}. The empty multisegment is denoted by $\boldsymbol{\emptyset}$
and we write $\mathfrak{M}$ for the set of all multisegments.
\end{definition}

For $e\in\mathbb{Z}_{\geq2}$, let us define $\mathfrak{M}_{e}$ as the subset
of $\mathfrak{M}$ of the multisegments $\mathfrak{m}$ in which each segment
$[a,b]$ is such that $1\leq a<b\leq e-1$. Also define $\mathfrak{M}%
_{e}^{\mathrm{aff}}$ as the subset of aperiodic multisegments of
$\mathfrak{M}$, that is the subset of multisegments $\mathfrak{m}%
\in\mathfrak{M}$ such that for each length $l$ there exists at least an
integer in $\{0,\ldots,e-1\}$ for which $\mathfrak{m}$ does not contain a
segment $[b-l+1,b]$ of length $l$ with $b=i\operatorname{mod}e$. It is then
known that in types $A_{e-1},A_{\infty}$ and $A_{e-1}^{(1)}$, the crystal
$B(\infty)$ has a simple realization with $\boldsymbol{\emptyset}$ as highest
weight vertex and in which the vertices are parametrized by the multisegments
in $\mathfrak{M}_{e},\mathfrak{M}$ and $\mathfrak{M}_{e}^{\mathrm{aff}}$,
respectively. Also we determined in \cite{JL2} the corresponding embedding%
\[
\Pi_{\Lambda_{\mathbf{s}}}:B(S_{\boldsymbol{s}}(\boldsymbol{\emptyset
}))\hookrightarrow B(\infty)
\]
compatible with the Uglov realization of crystals for a multicharge
$\mathbf{s}\in\mathbb{Z}^{l}$ such that $0\leq s_{1}\leq\cdots\leq s_{l}<e$
(with $s_{1}\geq1$ in type $A_{e-1}$ and $e=\infty$ in type $A_{\infty}%
$).\ The embedding $\Pi_{\Lambda_{\mathbf{s}}}$ can be described as follows.
Take ${\boldsymbol{\lambda}}$ an $l$-partition regarded as a sequence of $l$
Young diagrams. Then associate to each row $\lambda_{i}^{k}$ of
${\boldsymbol{\lambda}}$ the segment $[a;b]$ where $a=1-i+s_{k}$ and
$b=\lambda_{i}^{k}-i+s_{k}$ are the contents of the leftmost and rightmost
boxes in $\lambda_{i}^{k}$ translated by $s_{k}$, respectively. The
multisegment $\Pi_{\Lambda_{\mathbf{s}}}({\boldsymbol{\lambda}})$ is then the
formal sum of the segments associated to each non empty row of
${\boldsymbol{\lambda}}$.

\begin{example}
Consider the multicharge $\mathbf{s}:=(4,5)$ and the bipartition
$(3.2.2,3.1)$. Write the associated Young diagrams and fill each box in
$\lambda^{(k)},k\in\{1,2\}$ with its content translated by $s_{k}$:
\begin{Young}
		$4$ & $5$ & $6$ \cr
		$3$ & $4$ \cr
		$2$ & $3$ \cr
	\end{Young}\begin{Young}
		$5$ & $6$ & $7$ \cr
		$4$  \cr
	\end{Young}

Then, we have
\[
\Pi_{\Lambda_{\mathbf{s}}}=[4;6]+[3;4]+[2;3]+[5;7]+[4].
\]

\end{example}

In \cite{JL2}, we also got the action of the $\ast$-crystal operators and a
procedure to compute the minimal symbol associated to a multisegment (i.e. the
associated symbol with multicharge corresponding to $\varepsilon^{\ast
}(\mathfrak{m})$). Thus, by the previous arguments, we can use the results of
Section \ref{Sect_affineA} to decide wether a multisegment $\mathfrak{m}$
belongs to $B(\infty).$ This is direct for types $A_{e-1}$ and $A_{\infty}$
but in type $A_{e-1}^{(1)}$, one needs the characterization of the Demazure
crystals in the Uglov realization (see Remark \ref{Rem_keyUglov}) in order to
use the embedding $\Pi_{\Lambda_{\mathbf{s}}}$.

\subsection{Multisegments associated to a $(e,\boldsymbol{s})$-core.}

Given a segment $\mathfrak{m}\in\mathfrak{M}_{e}^{\mathrm{aff}}$, we now give
a direct procedure deciding whether $\mathfrak{m}\in\Pi_{\Lambda_{\mathbf{s}}%
}(O(\mathbf{s},e))$ or not\footnote{It also hold in type $A_{e-1}$ and
$A_{\infty}$ by using the relevant set of multisegments.}, that is
characterizing the image of the Key map for the Demazure crystals
$B(\infty)_{e}$. To do this, It will be convenient to write our aperiodic
multisegments by gathering segments with the same right end as follows:
\[
\mathfrak{m}=\sum_{1\leq j\leq m}\sum_{1\leq i\leq r_{j}}[a_{j}^{i},b_{j}]
\]
where $m\in\mathbb{N}$ and where, for each $1\leq j\leq m$, we have $r_{j}%
\in\mathbb{N}$. We can also assume that $b_{1}\leq\cdots\leq b_{m}$ and that
for each $1\leq j\leq m$, we have $a_{j}^{1}\leq\cdots\leq a_{j}^{r_{j}}$. Our
algorithm (illustrated by the example below) decides if $\mathfrak{m}\in
\Pi_{\Lambda_{\mathbf{s}}}(O(\mathbf{s},\infty))$ and then construct
recursively $m+1$ sequences of segments $(L_{1}^{j},\ldots,L_{l}%
^{j}),j=0,\ldots,m$ starting from $(L_{1}^{0},\ldots,L_{l}^{0})=(\emptyset
,\ldots,\emptyset)$.

\begin{itemize}
\item If $r_{m}>l$ then the algorithm stops.\ Otherwise set
\[
L_{l}^{1}=([a_{m}^{r_{m}},b_{m}]),\ldots,L_{l-r_{m}+1}^{1}=([a_{m}^{1}
,b_{m}]),L_{l-r_{m}}^{1}=\emptyset,\ldots,L_{1}^{1}=\emptyset.
\]

\item More generally, assume we have the sequence $(L_{1}^{m-j},\ldots,$
$L_{l}^{m-j})$ and consider the segments $[a_{j}^{i},b_{j}]$ for
$i=1,\ldots,r_{j}$. When $r_{j}>l$ the algorithm stops. Otherwise, set
$L_{l+r_{j}-k}^{m-j+1}=L_{l+r_{j}-k}^{m-j}$ for $r_{j}<k\leq l$ and for each
$1\leq k\leq r_{j}$, $L_{l+r_{j}-k}^{m-j+1}$ is obtained by adding the segment
$[a_{j}^{k},b_{j}]$ at the beginning of the sequence $L_{l+r_{j}-k}^{m-j}$ if
this sequence is empty or its first segment $[a,b]$ is such that $a=a_{j}
^{k}+1$ and $b>b_{j}$. If not, the algorithm stops.
\end{itemize}

At the end of the procedure either the algorithm stops before all the segments
of $\mathfrak{m}$ have been considered and we then conclude $\mathfrak{m}%
\notin\Pi_{\Lambda_{\mathbf{s}}}(O(\mathbf{s},e)$ or we get a sequence
$(L_{1}^{m},\ldots,L_{l}^{m})$ of segments :
\[
L_{j}^{m}:=([\alpha_{1},\beta_{1}],\ldots,[\alpha_{p_{j}},\beta_{p_{j}}])
\]
Then we consider the symbol
\[
S_{j}=%
\begin{tabular}
[c]{p{0.5cm}p{0.5cm}p{0.5cm}p{0.5cm}p{0.5cm}p{0.5cm}p{0.5cm}}%
$\alpha_{1}$ & $\alpha_{2}$ & $\ldots$ & $\alpha_{p}$ &  &  &
\end{tabular}
\
\]
We have $\mathfrak{m}\in\Pi_{\Lambda_{\mathbf{s}}}(O(\mathbf{s},e))$ if and
only if for all $(i,j)\in\{1,\ldots,l\}^{2}$ we have $s_{i}-s_{j}=p_{i}-p_{j}%
$. Moreover, it is easy to see that the symbol we have constructed is nothing
but the symbol associated to $\Pi_{\Lambda_{\mathbf{s}}}^{-1}(\mathfrak{m})$.

More generally this algorithm shows when there exists $\mathbf{s}\in
\mathbb{Z}^{l}$ such that $\mathfrak{m}\in\Pi_{\Lambda_{\mathbf{s}}%
}(O(\mathbf{s},e))$. The proof of the the rightness of the algorithm is
straightforward. The algorithm simply construct if possible the symbol of a
multipartition which satisfies all the properties of being in $\Pi
_{\Lambda_{\mathbf{s}}}(O(\mathbf{s},e))$.

\begin{example}
Assume $e=\infty$ and consider the multisegment%
\[
\lbrack2]+[3]+[2,3]+[2,3]+[4]+[3,4]+[5,6]+[6,7]+[4,7]+[7,9]+[5,9]+[3,9]
\]
We take $\mathbf{s}=(0,2,4)$.

\begin{itemize}
\item We start with the segments $[7,9]$, $[5,9]$ and $[3,9]$ and we get
$L_{1}^{1}=([3,9])$, $L_{2}^{1}=([5,9])$ and $L_{3}^{1}=([7,9])$.

\item We then take the segments $[6,7]$ and $[4,7]$ and we get $L_{1}
^{2}=([3,9])$, $L_{2}^{2}=([4,7],[5,9])$ and $L_{3}^{2}=([6,7],[7,9])$.

\item We then take the segments $[5,6]$ and we get $L_{1}^{3}=([3,9])$,
$L_{2}^{3}=([4,7],[5,9])$ and $L_{3}^{3}=([5,6],[6,7],[7,9])$.

\item We then take the segments $[4]$ and $[3,4]$ and we get $L_{1}
^{4}=([3,9])$, $L_{2}^{4}=([3,4],[4,7],[5,9])$ and $L_{3}^{4}
=([4],[5,6],[6,7],[7,9])$.

\item We then take the segments $[3]$, $[2,3]$ and $[2,3]$ and we get
$L_{1}^{5}=([2,3],[3,9])$, $L_{2}^{5}=([2,3],[3,4],[4,7],[5,9])$ and
$L_{3}^{5}=([3],[4],[5,6],[6,7],[7,9])$.

\item We finally take the segment $[2]$ and we get $L_{1}^{6}=([2,3],[3,9])$,
$L_{2}^{6}=([2,3],[3,4],[4,7],[5,9])$ and $L_{3}^{6}
=([2],[3],[4],[5,6],[6,7],[7,9])$.
\end{itemize}

We see that all the properties are satisfied and thus that $\mathfrak{m}
\in\Pi_{\Lambda_{\mathbf{s}}} ( O(\mathbf{s},\infty))$, the associated
$3$-partition is
\[
(7.2,5.4.2.2,3.2.2.1.1.1).
\]

\end{example}

\end{document}